\documentclass[11pt]{amsart}

\usepackage{amssymb}
\usepackage{amsmath,amsthm,geometry,enumerate}
\usepackage[dvips]{graphicx}
\usepackage{psfrag}
\usepackage{amscd}
\usepackage{xspace}
\usepackage[all]{xy}

\normalfont
\newcommand{\RedefinitSymbole}[1]{%
  \expandafter\let\csname old\string#1\endcsname=#1 \let#1=\relax
\newcommand{#1}{\csname old\string#1\endcsname\,}%
}
\RedefinitSymbole{\forall}
\RedefinitSymbole{\exists}

\def\cf{\textit{cf.}\kern.3em}

\def\resp{\textit{resp.}\kern.3em}
\renewcommand{\k}{\kern2pt}

\numberwithin{equation}{section} 

\let\emptyset\varnothing
\DeclareMathOperator{\pic}{Pic}
\DeclareMathOperator{\rk}{rk}

\DeclareMathOperator{\Hom}{Hom}
\DeclareMathOperator{\Aut}{Aut}

\DeclareMathOperator{\spec}{Spec}

\DeclareMathOperator{\codim}{codim}

\newtheorem{construction}[equation]{Construction}
\newtheorem{proposition}[equation]{Proposition}
\newtheorem{theorem}[equation]{Theorem}
\newtheorem{corollary}[equation]{Corollary}
\newtheorem{lemma}[equation]{Lemma}
\theoremstyle{definition}
\newtheorem{definition}[equation]{Definition}

\newtheorem{remark}[equation]{\textbf{Remark}}
\newtheorem{notation}[equation]{\textbf{Notation}}

\newcommand{\m}[2]{$\mathcal{M}_{#1,#2}$}
\newcommand{\mb}[2]{$\overline{\mathcal{M}}_{#1,#2}$}
\newcommand{\azero}{$\overline{A_1}^{[n]}$}
\newcommand{\azerop}{$\overline{A_1}$}
\newcommand{\auno}{$\overline{A_2}$}
\newcommand{\adue}{$\overline{A_3}$}
\newcommand{\atre}{$\overline{A_4}$}
\newcommand{\mgn}{$\mathcal{M}_{g,n}$}
\newcommand{\mbgn}{$\overline{\mathcal{M}}_{g,n}$}
\newcommand{\mun}{$\mathcal{M}_{1,n}$}
\newcommand{\mbun}{$\overline{\mathcal{M}}_{1,n}$}

\newcommand{\virg}[1]{\textquotedblleft#1\textquotedblright}

\begin{document}

\date{}
\title{\textbf{Chen--Ruan Cohomology of \mun \space and \mbun}}
\author{Nicola Pagani}
\subjclass[2000]{Primary 14H10 14N35 55N32; Secondary 14D23 14H37 55P50}
\begin{abstract}In this work we compute the Chen--Ruan cohomology and the stringy Chow ring of the moduli spaces of smooth and stable $n$-pointed curves of genus $1$. In the first part of the paper we study and describe stack theoretically the twisted sectors of \mun \ and \mbun. In the second part, we study the orbifold intersection theory of \mbun. We suggest a definition for an Orbifold Tautological Ring in genus $1$, which is both a subring of the Chen--Ruan cohomology and of the stringy Chow ring.\end{abstract}
\maketitle
\setcounter{tocdepth}{2}
\tableofcontents

\section {Introduction}

Motivated by physics, Chen--Ruan cohomology was introduced in the paper \cite{chenruan} in the analytic category, and in the two papers by Abramovich--Graber--Vistoli \cite{agv1} and \cite{agv2} in the algebraic category. This has produced two parallel objects: the Chen--Ruan cohomology and the stringy Chow ring, which provide the basis to develop the quantum cohomology ring of an orbifold. This cohomology ring recovers as a subalgebra the ordinary rational cohomology ring of the topological space that underlies the orbifold. As a vector space, the Chen--Ruan cohomology of $X$ is simply the cohomology of the inertia stack of $X$. If $X$ is an orbifold, its \emph{inertia stack} $I(X)$ is, loosely speaking, constructed as the disjoint union, for $g$ in the stabilizer of some point $x$ of $X$, of the locus stabilized by $g$ in $X$ (see Definition \ref{definertia}). As an example, the orbifold $X$ itself appears as a connected component of $I(X)$, as the locus fixed by the identity automorphism, which is trivially in the stabilizer group of every point. All the other connected components of the inertia stack $I(X)$ are usually called \emph{twisted sectors}. In this paper we use the algebraic language, and whenever the word \virg{orbifold} is mentioned, it stands for smooth Deligne--Mumford stack.

Among the first examples of smooth Deligne--Mumford stacks in the literature there are the moduli of smooth pointed curves \mgn \ and their compactifications \mbgn. It seems thus interesting to study their Chen--Ruan cohomology. This has been done so far for \mb{1}{1} (a special case of weighted projective space) and for $\mathcal{M}_{2}$ and $\overline{\mathcal{M}}_2$ by Spencer \cite{spencer} (see also \cite{spencer2}).

In the present work, we investigate the Chen--Ruan cohomology ring for $\mathcal{M}_{1,n}$ and $\overline{\mathcal{M}}_{1,n}$ with rational coefficients, assuming knowledge of the cohomology of \mun \ and \mbun. We show how it is possible to describe the stringy Chow ring in a similar fashion. Indeed we show that for each twisted sector, the cycle map from the Chow ring to cohomology is an isomorphism.  

The main results of this paper are the complete stack-theoretic description of the twisted sectors, and the explicit computation of the Chen--Ruan product as an extension of the usual cup product.

\begin {theorem} \label{first} (Theorem \ref{twistedcompact}, Corollary \ref{fondamentale}) Each twisted sector of \mbun \space is isomorphic to
\begin{displaymath}
A \times \overline{\mathcal{M}}_{0,n_1} \times \overline{\mathcal{M}}_{0,n_2} \times \overline{\mathcal{M}}_{0,n_3} \times \overline{\mathcal{M}}_{0,n_4},
\end{displaymath}

\noindent where the $n_i \geq 3$ are integers and $A$ is in the set
\begin{displaymath}
\{B \mu_3, B \mu_4, B \mu_6, \mathbb{P}(4,6), \mathbb{P}(2,4), \mathbb{P}(2,2)\}.
\end{displaymath}
Here $B G$ is the classifying stack of principal $G$-bundles, and $\mathbb{P}(a,b)$ is a weighted projective stack.
\end {theorem}

\begin {theorem} \label{second} (Theorem \ref{fundamental}) The Chen--Ruan cohomology ring of \mbun \ is generated as an algebra over the ordinary cohomology ring of \mbun \ by the fundamental classes of the twisted sectors with explicit relations.
\end {theorem}

The theory developed in this manuscript can be seen as an extension of the study of the cohomological properties of the moduli spaces of curves, initiated by Harer and Mumford (\cite{mumford}) in the eighties. The latter framework has produced important results in algebraic geometry, topology, mathematical and theoretical physics, representation theory and number theory. On the other hand, our results are the first steps towards the study and the understanding of more recent topics of investigation, such as the stringy topology and the Gromov-Witten theory of the moduli spaces of curves (see, for example, \cite{alr}). 

This paper is part of the PhD thesis \cite{paganitesi}, where the case of genus $g$ bigger than $1$ is also discussed. However, we believe that the genus $1$ case can be conveniently described within a more explicit and elementary framework, mainly thanks to the fact that the automorphism groups of stable genus $1$ marked curves are cyclic. In general we make an effort to limit to a minimum the use of technicalities, whereas various aspects of the theory could be developed in bigger generality to produce higher genera results, see \cite{pagani2}, \cite{paganihyper} and \cite{paganitommasi}.


%
\subsection{Description of the sections}

In Section $2$ we recall some known results that we will use and fix our notation. The complete, stack-theoretic description of the twisted sectors of \mun \ and \mbun \ is given in Section $3$, where we prove the first main result \ref{first}. This result allows us to compute the generating series of the orbifold Poincar\'e polynomials for \mbun. In Section $4$, we compute the Chen--Ruan cohomology of \mun \ and \mbun \ as a graded vector space. To do so, we introduce the unconventional rational grading on the cohomology of the inertia stack, usually referred to as \emph{age}, or \emph{degree shifting number}, or \emph{fermionic shift}. In Section $5$, we describe the twisted sectors of the second inertia stack of \mun \ and \mbun. Here a simplification occurs, indeed we show that every double twisted sector is canonically isomorphic to a twisted sector. In Section $6$ we begin the study of the orbifold intersection theory on \mbun, we compute all excess intersection bundles, and their top Chern classes. Finally, in Section $7$, we determine the Chen--Ruan cup product and we prove the second main result \ref{second}. In this section a proposal for an orbifold tautological ring for \mbun \ is motivated and advanced.

\section {Foundation}

\subsection{General notation}
We work in the category of schemes of finite type over $\mathbb{C}$. Although we treat only this case, some of our results can easily be extended to the case of an arbitrary field of characteristic different from $2$ and $3$. In the paper, algebraic stack means Deligne--Mumford stack. Intersection theory on schemes is defined in \cite{fulton}, on Deligne--Mumford stacks it is defined in \cite{vistoli-intersection}. We refer to these texts for the definitions and first properties of the Chow groups $A_*$. In this manuscript we work with cohomology and Chow ring with rational coefficients. Since all the spaces we consider are smooth, there is a standard identification of $A^*$ with the dual of $A_*$.  

 We use $\mathbb{G}_m$ to denote the group scheme of invertible multiplicative elements of $\mathbb{C}$. The discrete group subscheme of $\mathbb{G}_m$ of the $N$-th roots of $1$ is called $\mu_N$. The generators of $\mu_2$, $\mu_4$ and $\mu_6$ are
conventionally chosen to be respectively $-1, i$ and $\epsilon$. Since we work over the complex numbers, we can fix $\epsilon= e^{\frac{2 i \pi}{6}}$. We denote by $S_n$ the group of permutations of $[n]:= \{1,2,\ldots,n\}$.

If $G$ is a finite abelian group, $G^{\vee}=\Hom(G,\mathbb{C}^*)$ is the group of characters of $G$. We call the trivial $G$-gerbe over a point $BG=[ \spec (\mathbb{C} ) / G]$; it is the classifying stack of the group $G$. If $X$ is a scheme, an element in $\pic(X\times BG)$ is a pair $(L,\chi)$ where $L \in \pic(X)$ and $\chi \in G^{\vee}$.
\label{sezionegerbe} For a complete treatment on gerbes, we refer to the book \cite{giraud}, or to \cite{breen}. 

\subsection {Notation for \mgn \ and \mbgn, and some cohomological results}
\label{section2b}
 
 If $I$ is a finite set \mb{g}{I} \ is the moduli stack of stable genus $g$ curves with marked points in the set $I$. If $I=\{i, \bullet \}$ then we conventionally define $\overline{\mathcal{M}}_{0,I}$ as a point labeled by $i$. When the set $I$ is $[n]$, the set of the first $n$ natural numbers, we write \m{g}{n} \ and \mb{g}{n} \  instead of \m{g}{[n]} \ and \mb{g}{[n]}.

 If $I \subset J$, then $\pi_I:$ \m{g}{J} $\rightarrow$ \m{g}{I} \space is the morphism that remembers only the sections inside $I$. We give the same name to the morphism $\pi_I:$ \mb{g}{J} $\rightarrow$ \mb{g}{I} \space that forgets all the sections but the ones in $I$ and stabilizes. With this notation, let $s_i$ be the $i$-th section of $\pi_{[n]}:$ \mb{g}{n+1} $\rightarrow $   \mb{g}{n}. By definition, the \emph{cotangent line bundle} $\mathbb{L}_i$ is the line bundle $s_i^*(\omega_{\pi_{[n]}})$, where $\omega_{\pi_{[n]}}$ is the relative dualizing sheaf. We also define $\psi_i:=c_1(\mathbb{L}_i)$.

Let $k >0$ and let $(I_1,\ldots,I_k)$ be a partition of $[n]$. We define $j_{g,k}$ as the morphism gluing the marked points labeled with the same symbol:

\begin{displaymath}
j_{g,k}:\overline{\mathcal{M}}_{g, \coprod_{i=1}^k \bullet_i} \times \overline{\mathcal{M}}_{g_1,I_1 \sqcup \bullet_1} \times \ldots \times \overline{\mathcal{M}}_{g_k,I_k \sqcup \bullet_k} \to \overline{\mathcal{M}}_{g+\sum g_i,n},
\end{displaymath}
note that $j_{g,k}$ depends on the partition $I_1, \ldots, I_k$ and on the choice of $g_1, \ldots, g_k$, although we do not make this explicit in our notation.
We also define:
\begin{displaymath}
j: \overline{\mathcal{M}}_{g,n \sqcup \bullet_1 \sqcup \bullet_2} \to \overline{\mathcal{M}}_{g+1,n}
\end{displaymath}
as the morphism gluing together $\bullet_1$ and $\bullet_2$. 
In this paper, we will be dealing with the case of genus $1$ curves. We will be using several times the map $j_{1,k}$, where all the $g_i$ are set equal to $0$. We call this map simply $j_k$, so if $I_1, \ldots, I_k$ is a partition of $[n]$ we have the gluing map:
\begin{displaymath}
j_k:\overline{\mathcal{M}}_{1, \coprod_{i=1}^k \bullet_i} \times \overline{\mathcal{M}}_{0,I_1 \sqcup \bullet_1} \times \ldots \times \overline{\mathcal{M}}_{0,I_k \sqcup \bullet_k} \to \overline{\mathcal{M}}_{1,n}.
\end{displaymath}
The product space on the left admits projection maps onto each factor. We denote by $p$ the projection map onto the first factor $p$, and with $p_i$ the projection maps onto the genus $0$ component with marked points in the set $I_i$.

We recall the definition and main properties of the tautological ring for the moduli spaces of curves. 
\begin {definition} (\cite[Section 0.1]{faberpanda}) \label{tautologico}
The \emph{system of tautological rings} $R^*$(\mb{g}{n}$)$ is defined to be the set of smallest $\mathbb{Q}$-subalgebras of the Chow rings
\begin{displaymath}
R^*_{g,n}= R^*(\overline{\mathcal{M}}_{g,n}) \subset A^*(\overline{\mathcal{M}}_{g,n}, \mathbb{Q})
\end{displaymath}
that is closed under push--forward via all forgetful and gluing maps.
\end {definition}

\begin {remark} \label{nonloso} The system of tautological rings is closed under pull--back via the forgetful and the gluing maps. Each tautological ring is an $S_n$-representation via the action that permutes the points. We denote by $RH^*($\mb{g}{n}$)$ the image of $R^*($\mb{g}{n}$)$ under the cycle map to the ring of even cohomology classes. 
\end {remark}

\begin {definition} \label{boundary} We define $B^*_{g,n}$ to be the smallest system of vector subspaces of the Chow rings $A^*(\overline{\mathcal{M}}_{g,n}, \mathbb{Q})$ that contain the fundamental classes, and that are stable under push--forward via all gluing maps (see Definition \ref{tautologico}). A \emph{boundary strata class} is an element in $B^*_{g,n}$ that corresponds to a closed irreducible proper substack of $\overline{\mathcal{M}}_{g,n}$.
\end {definition}
\noindent Obviously, the tautological ring contains all boundary strata classes.

\begin{notation} \label{divisorigenere0e1} 
If $I\subset [n], |I|\geq2$, we denote by $d_I$ the closure of the substack of \mb{1}{n} of reducible nodal curves with two smooth components, where the marked points in the set $I$ are on the genus $0$ component and the marked points on the genus $1$ curve are in the complement. The closure of the substack of \mb{1}{n} of irreducible curves of geometric genus $0$ is called $d_{irr}$. We will sometimes indicate by $d_I$ also the class $[d_I] \in H^2($\mb{1}{n}$)$ represented by the divisor $d_I$. These elements form a basis for $B^1_{1,n}$.

Analogously, given $I \subset [n]$, such that $|I| \geq 2$ and $|[n] \setminus I| \geq 2$, $\Delta_{I}= \Delta_{[n] \setminus I}$ is the sublocus of $\overline{\mathcal{M}}_{0,n}$ whose general element has two genus $0$ components with marked points in $I$ in the first one and in $[n] \setminus I$ in the second one.  These elements generate $B^1_{0,n}$.
\end{notation}

\noindent In general, $B^*_{g,n} \subset R^*_{g,n} \subset A^*(\overline{\mathcal{M}}_{g,n}, \mathbb{Q}) \to H^*(\overline{\mathcal{M}}_{g,n}, \mathbb{Q})$ are all distinct. By \cite[p.2]{faberpanda}, the $\psi$-classes defined in Section \ref{section2b} are in the tautological ring. In genus $0$ we have the equalities \begin{equation} \label{skeel} B^*_{0,n}= R^*_{0,n}= A^*(\overline{\mathcal{M}}_{0,n}, \mathbb{Q})= H^{2*}(\overline{\mathcal{M}}_{0,n},\mathbb{Q});\end{equation} moreover the cohomology is generated by the boundary divisor classes (see \cite{keel}). In the remainder of this section, we will see what we can say for the cohomology of \mbun, in analogy with \eqref{skeel}. The next proposition is a straightforward consequence of Theorem $*$, \cite[Theorem 1.1]{grabervakil}. 
\begin {proposition} \label{conseteoremstar} The tautological ring $R^*($\mb{1}{n}$)$ is additively generated by boundary strata classes (see \ref{boundary}), so that $B^*_{1,n}=R^*_{1,n}$.
\end {proposition}
So let us now introduce some results concerning the tautological ring of \mbun. These results were originally claimed by Getzler in \cite{getzler1}, and they were recently proved by Petersen in \cite{petersen}.

\begin {theorem} \label {getzlerclaim} (claimed for the first time in \cite {getzler1}, proved in \cite{petersen}) The boundary strata classes span the even cohomology of \mbun. Moreover the cycle map is injective when restricted to the tautological algebra $R^*($\mbun$)$.
\end {theorem}
Note that the second sentence of the statement follows, in Petersen's proof, from the stronger fact that all relations among the generators of the even cohomology come from genus $0$ relations, and from Getzler's relation \cite[Theorem 1.8]{getzler1}, see \cite{petersen} for more details. The tautological ring of \mbun \ has been studied in detail by Belorousski \cite{belo}. When $n\leq 10$, the picture is similar to the genus $0$ case \eqref{skeel}.

\begin{proposition} \label{nminoredieci} (\cite[Theorem 3.1.1, Theorem 3.6.3]{belo}) For $n \leq 10$ the following two equalities also hold 
\begin{displaymath} R^*_{1,n}=A^*(\overline{\mathcal{M}}_{1,n}, \mathbb{Q})=H^*(\overline{\mathcal{M}}_{1,n}, \mathbb{Q}).\end{displaymath}
\end {proposition}

\noindent It is well known that the eleventh cohomology group of $\overline{\mathcal{M}}_{1,11}$ is non-zero. It follows that the second and third equalities of the proposition above are no longer true for $n \geq 11$ (see for instance \cite[p.2]{graberpanda}).

Anyway, after Proposition \ref{conseteoremstar}, and Theorem \ref{getzlerclaim}, we obtain a decomposition of the cohomology in boundary strata classes and odd cohomology
\begin{equation} \label{decomposizione}
H^*(\overline{\mathcal{M}}_{1,n})= B^*_{1,n} \oplus H^{odd}(\overline{\mathcal{M}}_{1,n}).
\end{equation}
The rings $R^*_{1,n}=B^*_{1,n}$ are, in general, not multiplicatively  generated by the boundary divisors (as it happens in genus $0$), by a result of Belorousski.
\begin {theorem}\cite[Chapter 3]{belo} \label{belotheorem} The Chow ring of \mbun \ is generated by the divisors precisely when $n \leq 5$.
\end{theorem}
\noindent Anyway, it is possible to give a simple and geometric description of the additive generators of $R^*_{1,n}$ besides the product of boundary divisors. We denote by $R^{div}_{1,n}$ the subring of $R^*_{1,n}$ generated by the classes of the divisors. 
\begin{definition} The \emph{banana locus} (cf. \cite[p.49]{belo})  is the locus in \mbun \ of curves whose general element has two rational components joined in two nodes. A \emph{subbanana cycle} is a boundary strata class in \mbun \ that is contained in the banana locus. Let $R^{ban}_{1,n}$ be the vector subspace of $R^*_{1,n}$ generated by subbanana cycles. \label{bananacycle}
\end{definition}
\noindent Then we have\footnote{We learned this from Belorousski's thesis \cite{belo}, although it is not explicitly written there.}
\begin{equation} \label{sub}
R^*_{1,n}=B^*_{1,n}=R^{div}_{1,n} + R^{ban}_{1,n}.
\end{equation}
Indeed, let us consider the complement of the banana locus in \mbun. Reasoning by induction on the codimension, it is simple to see that the boundary strata classes on this complement can all be written as products of divisor classes.

\section {The Chen--Ruan cohomology of $\mathcal{M}_{1,n}$ and $\overline{\mathcal{M}}_{1,n}$ as vector spaces}
\subsection{Definition of Chen--Ruan cohomology as a vector space}

The following is a natural stack associated to a stack $X$, which points to where $X$ fails to be an algebraic space.

\begin{definition} \label{definertia} (\cite[Definition 1.12] {vistoli-intersection}) Let $X$ be an algebraic stack. The \emph{inertia stack} $I(X)$ of $X$ is defined as the fiber product $X \times_{X \times X} X$ where both morphisms $X \rightarrow X \times X$ are the diagonal morphisms. There is a natural map $f: I(X) \to X$.
\end {definition}

\noindent The construction of Chen--Ruan cohomology based on the definition of inertia orbifold was given for the first time in \cite[Definition 3.2.3]{chenruan}. As observed in \cite[Section 4.4]{agv1}, the latter is nothing but the coarse moduli space of the inertia stack we have just introduced. In \cite[7.3]{agv2} the algebraic counterpart of Chen--Ruan cohomology is introduced, under the name of stringy Chow ring. It is built on the rigidification of the cyclotomic inertia stack introduced in \cite[Section 3]{agv2}. In this paper we work over $\mathbb{C}$, and all cohomologies are taken with rational coefficients. Therefore, the cohomologies of the inertia stack, of the cyclotomic inertia stack, of the rigidified cyclotomic inertia stack and of the inertia orbifold are all canonically isomorphic, since all of them share the same coarse moduli space.

\begin{remark} \label{closedsubstack} If $Y$ is a twisted sector of $I(X)$, then the map $f$ of Definition \ref{definertia} restricts to a map $f_{|Y}:Y \to X$. In general $f_{|Y}:Y \to X$ is a composition of a stack covering and a closed embedding, as easily follows for instance from \cite[Lemma 1.13]{vistoli-intersection}, or from \cite[3.1.3]{chenruan}. In the present paper however, since all the stacks we consider are abelian orbifolds, the map $f_{|Y}$ is a closed embedding. So if $Y$ is a twisted sector, it can be written as $Y=(Z,g)$, where $Z$ is a closed substack of $X$ and $g$ is an automorphism in the generic stabilizer of $Z$. 
\end{remark}

\begin {definition} If $X$ is an algebraic stack, the connected component of the inertia stack associated with the identity automorphism is called the \emph{untwisted sector} of the inertia stack. All the remaining connected components are called the \emph{twisted sectors} of $I(X)$. The latter are sometimes called also the \emph{twisted sectors} of $X$.
\end{definition} 

\begin{proposition} \label{liscezza1} \cite[Corollary 3.1.4]{agv2} Let $X$ be a smooth algebraic stack. Then the inertia stack $I(X)$ is smooth.
\end{proposition}

\begin {definition}\label{decompo} Let $X$ be a smooth algebraic stack. Let $T$ be a set of indices in bijection with the twisted sectors of $I(X)$. We say that the equality
\begin{displaymath}
I(X)= X \sqcup \coprod_{i \in T} (X_i,g_i),
\end{displaymath}
\noindent is a \emph{decomposition of the inertia stack of $X$ in twisted sectors}. 
\end{definition}

\begin {notation} \label{notatio} In order to simplify the notation, if $(A,g)$, $(A,g')$ are two twisted sectors, we shall write $(A,g/g')$ to denote the disjoint union of the two twisted sectors $(A,g)$ and $(A,g')$ in the inertia stack. When we write $A$ we refer to the image of the closed embedding of the twisted sector inside the original stack $X$ (see Remark \ref{closedsubstack}).
\end{notation}
We can then define the Chen--Ruan cohomology vector space.
\begin {definition} \label{crvector} (\cite[Definition 3.2.3]{chenruan}) Let $X$ be a smooth algebraic stack. Then the Chen--Ruan cohomology is by definition
\begin{displaymath}H^*_{CR}(X,\mathbb{Q}):=H^*(I(X),\mathbb{Q})\end{displaymath}
as a rational vector space.
\end {definition}
\noindent The Chen--Ruan cohomology decomposes as in Definition \ref{decompo}
\begin{displaymath}
H^*_{CR}(X, \mathbb{Q})= H^*(X, \mathbb{Q}) \oplus \bigoplus_{i \in T} H^*(X_i, \mathbb{Q}).
\end{displaymath}

\subsection {The inertia stack of \mun \ and \mbun}

The twisted sectors in case $n=1$ are well known as a direct consequence of the Weierstrass Theorem. We refer to \cite[III.1]{silverman} for the basic material on this topic. First of all, recall that every curve of the form: \begin{displaymath}{C}_{a,b}= \{[x:y:z]| \ z y^2=x^3+a z^2x + bz^3, \ \Delta:=4a^3+27b^2 \neq 0 \} \subset \mathbb{P}^2 \end{displaymath}
is a smooth genus $1$ curve. If, instead: \begin{displaymath}{C}_{a,b}= \{[x:y:z]| \quad z y^2=x^3+a z^2x + bz^3, \ \Delta:=4a^3+27b^2 = 0, (a,b) \neq (0,0) \},\end{displaymath}
then ${C}_{a,b}$ is a nodal curve of arithmetic genus $1$, geometric genus $0$ and one node. All genus $1$ curves with a marked point admit this description.
 
\begin {theorem} \cite[III.1]{silverman} \label{weierstrass} (Weierstrass representation)  Let $(C,P)$ be an elliptic curve, possibly nodal. Then there exist $(a,b) \in \mathbb{C}^2$ such that $(C,P)$ is isomorphic to $(C_{a,b},[0:1:0])$, where $C_{a,b}$ is as above.
If $\alpha$ is an isomorphism of $(C,P)$ with $(D,Q)$ then there exists $\lambda \in \mathbb{G}_m$ such that, up to the isomorphism above, $\alpha$ is
\begin{displaymath}\alpha: \begin{cases} a \rightarrow \lambda^4 a\\
b \rightarrow \lambda^6 b\\
x \rightarrow \lambda^2 x\\
y \rightarrow \lambda^3 y \\
z \rightarrow z. \end{cases}\end{displaymath}
\end {theorem}
\noindent From this it follows that the moduli stack $\overline{\mathcal{M}}_{1,1}$ is isomorphic to the weighted projective stack $\mathbb{P}(4,6)$.

\begin{notation} \label{c4c6}
There are two elements of \mb{1}{1} that are stabilized by the action of a group respectively isomorphic to $\mu_4$ and $\mu_6$, we call them respectively $C_4$ and $C_6$. These are classes of curves whose Weierstrass representation can be chosen respectively as:
\begin{displaymath}
\mathcal{C}_4:=\left\{[x:y:z]|\quad y^2 z=x^3+x z^2 \right\} \subset \mathbb{P}^2,
\end{displaymath}
 \begin{displaymath}\mathcal{C}_6:=\left\{[x:y:z]|\quad y^2 z=x^3+z^3 \right\} \subset \mathbb{P}^2.\end{displaymath}
\end{notation}
 If $(C,P)$ is an elliptic curve, and $G$ is its automorphism group, then it can be identified canonically with $\mu_N$ for a certain $N \in \{2,4,6\}$.
 
\begin{notation} \label{canonical} If $(C,P)$ is an elliptic curve, and $G$ is its automorphism group, then $G$ acts effectively on $T^{\vee}_P(C)$, the cotangent space in $C$ to $P$, which is canonically isomorphic to $\mathbb{C}$. We identify $G$ with $\mu_N$ under this isomorphism.
\end{notation}

The decomposition of the inertia stack of \m{1}{1} and \mb{1}{1} in twisted sectors (Definition \ref{decompo}), is a simple way to summarize the well--known facts that we have exposed in this section.
 
\begin {corollary} \label{inertiam11} With the notation introduced in Notation \ref{notatio} and \ref{canonical}, the decomposition of the inertia stack of \m{1}{1} in twisted sectors is:
\begin{displaymath}I(\mathcal{M}_{1,1})=(\mathcal{M}_{1,1}, 1) \sqcup (\mathcal{M}_{1,1}, -1) \sqcup (C_4, i/-i) \sqcup (C_6, \epsilon/\epsilon^2/\epsilon^4/ \epsilon^5)\end{displaymath}
and that of \mb{1}{1} is:
\begin{displaymath}I(\overline{\mathcal{M}}_{1,1})=(\overline{\mathcal{M}}_{1,1}, 1) \sqcup (\overline{\mathcal{M}}_{1,1}, -1) \sqcup (C_4, i/-i) \sqcup (C_6, \epsilon/\epsilon^2/\epsilon^4/ \epsilon^5).\end{displaymath}
\end {corollary}

\subsubsection{The case of \mun}

We now study the inertia stack of \mun. Note that if $n >4$, the objects of \mun \ are rigid, and therefore in that range $I(\mathcal{M}_{1,n})=\mathcal{M}_{1,n}$. A simple analysis of the fixed points of the action of $\mu_3$, $\mu_4$ and $\mu_6$ on the curves $\mathcal{C}_4$ and $\mathcal{C}_6$ (see Notation \ref{c4c6}) by Theorem \ref{weierstrass} leads to three special points of \mbun, $n \leq 3$.

\begin {notation} \label{notsmooth} We call the point in \m{1}{2} stabilized by $i$ or $-i$ $C_4'$, the point in \m{1}{2} stabilized by $\epsilon^2$ or $\epsilon^4$ $C_6'$, and the point in \m{1}{3} stabilized by $\epsilon^2$ or $\epsilon^4$ $C_6''$.
\end {notation}

To complete the study of the loci fixed by automorphisms in \mun, we shall need the loci fixed by the elliptic involution (according to Notation \ref{canonical} we write it as $(-1)$). We give a special name to them.

\begin {definition} \label{aii} Let $1 \leq i \leq 4$. We define $A_i$ as the closed substack of \m{1}{i} whose objects $A_i(S)$ are $i$-marked smooth genus $1$ curves over $S$ such that the sections are stabilized by the elliptic involution.
\end {definition}

\noindent We shall see, as a consequence of Lemma \ref{aicompact}, that $A_i$ is connected for all $i$ (note that $A_1= \mathcal{M}_{1,1}$). What we have just discussed, leads to the following description:

\begin{corollary}\label{twistedsmooth} The decomposition of the inertia stack of \mun \ (Notation \ref{decompo}, \ref{notatio}, \ref{canonical}) is:
\begin{itemize}
\item $I(\mathcal{M}_{1,1})=  \mathcal{M}_{1,1} \sqcup (\mathcal{M}_{1,1}, -1) \sqcup (C_4,i/-i) \sqcup (C_6,\epsilon/ \epsilon^2/ \epsilon^4/ \epsilon^5)$;
\item $I(\mathcal{M}_{1,2})= \mathcal{M}_{1,2} \sqcup (A_2, -1) \sqcup (C_4',i/-i) \sqcup (C_6', \epsilon^2/\epsilon^4)$;
\item $I(\mathcal{M}_{1,3})=\mathcal{M}_{1,3} \sqcup (A_3, -1) \sqcup (C_6'', \epsilon^2/ \epsilon^4)$;
\item $I(\mathcal{M}_{1,4})=\mathcal{M}_{1,4} \sqcup (A_4, -1)$;
\item $I(\mathcal{M}_{1,n})= \mathcal{M}_{1,n}$ if $n \geq 5$.
\end{itemize}
\end{corollary}

\noindent We collect the twisted sectors of \m{1}{n} in the following table. Different rows correspond to different automorphisms, while the $i$-th column corresponds to the twisted sectors of \m{1}{i}. 
\ \\
\ \\
{\centerline { \begin {tabular}{|c|c|c|c|c|}
\hline
&$1$&$2$&$3$&$4$ \\
\hline
$-1$ & $A_1$ & $A_2$ & $A_3$ & $A_4$ \\
\hline
$\epsilon^2/\epsilon^4$ & $C_6$ & $C_6'$& $C_6''$ & $\emptyset$ \\
\hline
$i/-i$ & $C_4$ & $C_4'$  & $\emptyset$ & $\emptyset$ \\
\hline
$\epsilon / \epsilon^5$ & $C_6$ & $\emptyset$ & $\emptyset$ & $\emptyset$ \\ 
\hline
\end {tabular}}
}

\ \\

 We now investigate the geometry of the spaces $A_i$ introduced in Definition \ref{aii}. In particular, this will give their cohomology.

\begin{remark} Using analytic methods (see \cite[Chapter 3]{diamond}), it is known that the coarse moduli space of $A_i$ is a genus $0$ quasiprojective curve, and also how many points are needed to compactify it. In the literature, the coarse moduli space for $A_2$ is known under the name of $X_1(2)=X_0(2)$. The coarse moduli space for $A_3\cong A_4$ is usually called $X(2)$. We here want to give an algebraic and stack-theoretic description of those spaces, that we could not find anywhere. 
\end{remark}
 
\begin {definition} We define $\overline{A_i}$ as the closure of $A_i$, in \mb{1}{i}.
\end {definition}

\noindent We have already observed that the stack $\overline{A_1}\cong \overline{\mathcal{M}}_{1,1}$ is isomorphic to $\mathbb{P}(4,6)$ as a consequence of Theorem \ref{weierstrass}. Following the same strategy that can be used to prove the latter isomorphism, we can obtain the following result:

\begin{lemma} \label{aicompact} The stack \auno \ is isomorphic to the weighted projective stack $\mathbb{P}(2,4)$. The stacks \adue \ and \atre \ are isomorphic to the weighted projective stack $\mathbb{P}(2,2)$.
\end{lemma}

\begin {proof} We first study the case of $\overline{A_2}$. Let us define the following space:
\begin{displaymath}B_1:= \left\{ \left((a,b),[x:y:z]\right) \ | \ (a,b) \neq (0,0), \ zy^2=x^3+a z^2 x+b z^3\right\} \subset \mathbb{A}^2_0 \times \mathbb{P}^2.\end{displaymath}
The projection onto the first factor, with the section $\sigma_1(a,b):= ((a,b),[0:1:0])$, describes this space as an elliptic fibration over $\mathbb{A}^2_0$, so it determines a unique map $\phi:\mathbb{A}^2_0 \to$ \mb{1}{1},  (here $\overline{\mathcal{C}}_{1,1}$ is the universal curve):
\begin{displaymath}
\xymatrix{B_1 \ar[r] \ar[d] & \overline{\mathcal{C}}_{1,1} \ar[d] \ar[r]^{\sim} & \overline{\mathcal{M}}_{1,2} \\ 
\mathbb{A}^2_0 \ar^{\phi}[r] \ar@/_1pc/[u]_{\sigma_1} & \overline{\mathcal{M}}_{1,1} \ar@/_1pc/[u]_{x_1}.&}
\end{displaymath}
making the diagram cartesian.
It is a well--known consequence of the Weierstrass theorem (\ref{weierstrass}) made in families that the map $\phi$ factors via the quotient $[\mathbb{A}^2_0 / \mathbb{G}_m]$, where $\mathbb{G}_m$ acts with weights $4$ and $6$, and that the resulting map $\tilde{\phi}: [\mathbb{A}^2_0 / \mathbb{G}_m] \to \overline{\mathcal{M}}_{1,1}$ is an isomorphism of stacks. The locus in $B_1$ cut out by the equation $y=0$ surjects onto $\overline{A_2} \subset \overline{\mathcal{M}}_{1,2}$. This locus is isomorphic to $\mathbb{A}^2_0$ with parameters $(a,x)$. Again as a consequence of the Weierstrass theorem, the action of $\mathbb{G}_m$ with weights $4$ and $2$ can be factored out, to obtain an isomorphism of stacks $[\mathbb{A}^2_0/\mathbb{G}_m] \to \overline{A_2}$. The forgetful map $\overline{A_2} \to \overline{A_1}$ lifts to the map of the charts $\mathbb{A}^2_0 \to \mathbb{A}^2_0$
\begin{equation} \label{forg1}
 (a,x) \to (a, - ax - x^3).
\end{equation}  

Now we study the case of $\overline{A_3}$. Let:
\begin{displaymath}B_2:= \left\{ \left((a,x_1),[x:y:z]\right) \ | \ (a,x_1) \neq (0,0), \ zy^2=x^3+a z^2 x+(-a x_1 -x_1^3) \ z^3\right\}.\end{displaymath}
In this case, the projection onto the first factor with the two sections \begin{displaymath} \begin{cases}\sigma_1(a,x_1)=((a,x_1),[0:1:0]) \\ \sigma_2(a,x_1)=((a,x_1),[x_1:0:1])\end{cases} \end{displaymath} does not give a map to \mb{1}{2} \ since the image of the second section intersects the singular locus. We define
\begin{displaymath}
\Lambda:= \left\{ ((a,x_1),[x:y:z]) \ | \ x=x_1, \ y=0, \ 4 a^3 + 27 (-a x_1 -x_1^3)^2 = 0\right\} \subset B_2.
\end{displaymath}
Let $p:\tilde{B_2} \to B_2$ be the blow-up of $B_2$ in $\Lambda$. Now the projection of $\tilde{B_2}$ onto $\mathbb{A}^2_0$ admits two distinct sections $\tilde{\sigma_1}$ and $\tilde{\sigma_2}$ that to not intersect the singular locus, and such that $p \circ \tilde{\sigma_i}=\sigma_i$. In this way, we obtain the cartesian diagram: 
\begin{displaymath}
\xymatrix{\tilde{B_2} \ar[r] \ar[d] & \overline{\mathcal{C}}_{1,2} \ar[d] \ar[r]^{\sim} & \overline{\mathcal{M}}_{1,3} \\ 
\mathbb{A}^2_0 \ar^{\psi}[r] \ar@/_1pc/[u]_{\tilde{\sigma_1}} \ar@/_2pc/[u]_{\tilde{\sigma_2}} & \overline{\mathcal{M}}_{1,2} \ar@/_1pc/[u]_{x_1}\ar@/_2pc/[u]_{x_2}.&}
\end{displaymath}
We denote by $\tilde{\psi}$ the smooth map $\tilde{B_2} \to \overline{\mathcal{M}}_{1,3}$.
Let $F$ be the locus in $B_2$ cut out by the equation $y=0$ and $\tilde{F}$ its strict transform under $p:\tilde{B_2} \to B_2$. The map $\tilde{\psi}$ restricted to $\tilde{F}$ surjects onto $\overline{A_3}$. There is an isomorphism $\lambda$ from $\mathbb{A}^2_0$ (parameters $(x_2,x_1)$) to $F$ \begin{displaymath}\lambda: (x_2,x_1) \to ((-x_1^2 -x_1 x_2 -x_2^2, x_1), [x_2:0:1]).\end{displaymath} Since $F$ is smooth, the restriction of the map $p: \tilde{F} \to F$ is an isomorphism, and therefore $\lambda$ lifts to an isomorphism $\tilde{\lambda}: \mathbb{A}^2_0 \to \tilde{F}$.  So we have a surjection:
 \begin{displaymath}\tilde{\psi} \circ \tilde{\lambda}: \mathbb{A}^2_0 \to \overline{A_3}.\end{displaymath}
Again, as a consequence of Weierstrass theorem, this map factors via the quotient $[\mathbb{A}^2/\mathbb{G}_m]$, where the action has weights $2$ and $2$, thus inducing an isomorphism of stacks $[\mathbb{A}^2_0/\mathbb{G}_m] \to \overline{A_3}$. The forgetful map $\overline{A_3} \to \overline{A_2}$ lifts to the map of the charts $\mathbb{A}^2_0 \to \mathbb{A}^2_0$
\begin{equation} \label{forg2}
 (x_1,x_2) \to (-x_1^2-x_1 x_2 -x_2^2, x_1).
\end{equation} 

To conclude the proof, we observe that the restriction of the forgetful map, $\overline{A_4} \to \overline{A_3}$, is an equivalence of categories. Indeed, when three among the four $2$-torsion points of an elliptic curve have been chosen, the fourth is uniquely determined.
\end{proof}

Note that as a consequence of the proof, we deduce from \eqref{forg1} and \eqref{forg2} a description of the forgetful maps $\overline{A_3} \to \overline{A_2} \to \overline{A_1}$ in terms of maps of weighted projective stacks. 

In Figure \ref{fig2} we show the two points in $\overline{A_2} \setminus A_2$, the three points in $\overline{A_3} \setminus A_3$ are in Figure \ref{Fig3}. The irreducible components of these curves are all rational: the geometric genus of each component is written at one extreme of the curve itself. The first marked point is not pictured and is at infinity. If coordinates are chosen on the vertical genus $0$ curve, in such a way that the two intersection points with the other component are $0, \infty$, the marked points ($2$ and) $2,3$ are (chosen among) the points with coordinates $1,-1$.

\begin{figure}[ht]
\centering \footnotesize
\psfrag{2}{$2$} 
\psfrag{0}{$0$}
\begin{tabular}{ccc}
\includegraphics[scale=0.15]{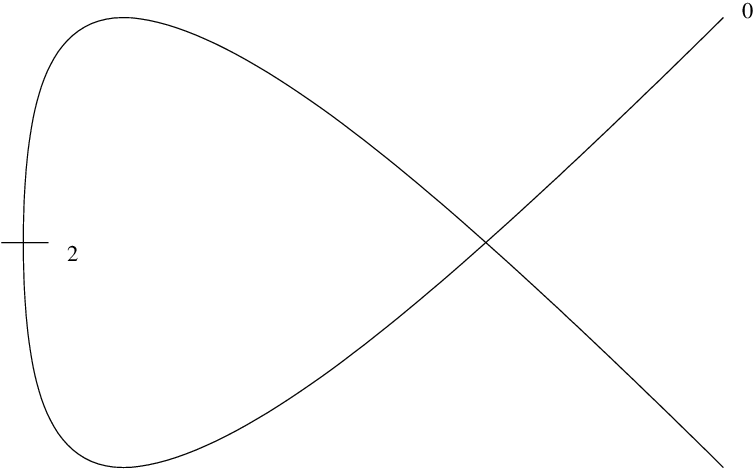} && \includegraphics[scale=0.25]{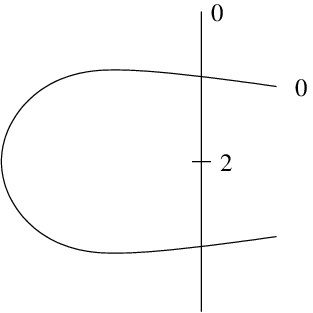} 
\end{tabular}
\caption{The two points that compactify $A_2$}
\label{fig2}
\end{figure}

\begin{figure}[ht]
\centering
\footnotesize
\psfrag{2}{$2$} 
\psfrag{0}{$0$}
\psfrag{3}{$3$}
\begin{tabular}{ccc}
\includegraphics[scale=0.30]{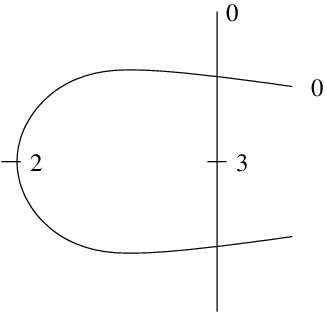}&
\includegraphics[scale=0.30]{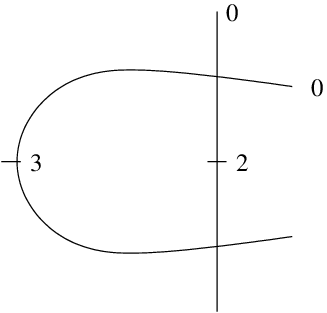}&
\includegraphics[scale=0.30]{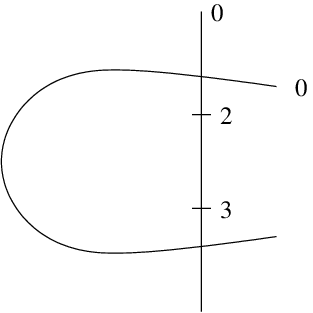}
\end{tabular}
\caption{The three points that compactify $A_3$}
\label{Fig3}
\end{figure}

We have thus described the spaces $A_i$ as open dense substacks of weighted projective stacks of dimension $1$. In particular, their rational cohomology follows from the following corollary.

\begin {corollary} The coarse moduli spaces of $\overline{A_i}$ is isomorphic to $\mathbb{P}^1$. If $i <4$, the coarse moduli spaces of $A_i$ is $\mathbb{P}^1$ minus $i$ points. The moduli stack $A_4$ coincides with $A_3$.
\end {corollary}

\subsubsection{The case of $\overline{\mathcal{M}}_{1,n}$}

Up to now, we have determined the decomposition in twisted sectors of the inertia stack (\ref{decompo}) of $\mathcal{M}_{1,n}$ (Corollary \ref{twistedsmooth}). As we have observed in Remark \ref{closedsubstack}, each twisted sector of \mb{1}{n} admits a closed embedding inside \mb{1}{n} itself. So, if $(Z,g)$ is a twisted sector of \m{1}{n}, we can consider the Deligne--Mumford compactification $\overline{Z} \subset$ \mb{1}{n}. It is easily seen that $(\overline{Z},g)$ is a twisted sector of \mb{1}{n}. In the last section, we have studied all such twisted sectors $(\overline{Z},g)$. We will see in this section that there are further twisted sectors of \mb{1}{n}, whose image inside \mb{1}{n} is completely contained in the boundary.

Let $(Z,g)$ be a twisted sector of \m{1}{k}. The twisted sector $(\overline{Z},g)$ has $k$ marked points that we rename $\bullet_i$, where $1 \leq i \leq k$. Let now $(I_1,\ldots,I_k)$ be a partition of $[n]$, such that $I_i\neq \emptyset$. Let now $j_k$ be the morphism gluing together the same symbols (defined in Section \ref{section2b}):

\begin {definition} \label{base} Let $Z$ be a twisted sector inside \m{1}{k}. We say $\overline{Z}$ is a \emph{base twisted sector}. We define $\overline{Z}^{(I_1,\ldots,I_k)}$ as:

\begin{displaymath}
\overline{Z}^{(I_1,\ldots,I_k)}:= j_k (\overline{Z} \times \overline{\mathcal{M}}_{0,I_1 \sqcup \bullet_1} \times \ldots \times \overline{\mathcal{M}}_{0,I_k \sqcup \bullet_k}).
\end{displaymath}
We then say $\overline{Z}$ is the \emph{base twisted sector} associated with $\overline{Z}^{(I_1,\ldots,I_k)}$.
\end {definition}

\begin {theorem} \label{partialtwistedcompact} If $(Z, \alpha)$ is a twisted sector in \m{1}{k}, and $(I_1,\ldots,I_k)$ is a partition of $[n]$, $(\overline{Z}^{(I_1,\ldots,I_k)},\alpha)$ is a twisted sector of the inertia stack of \mbun.
\end {theorem}

\begin {proof} The automorphism $\alpha$ lifts to an automorphism $\alpha'$ of $\overline{Z}^{(I_1,\ldots,I_k)}$ that acts as $\alpha$ on the base and as the identity on the components \mb{0}{n}. We can call with the same name $\alpha$ and $\alpha'$, and represent them by the same element in $\mu_N$ (see Notation \ref{canonical}). It is easy to check that $\overline{Z}^{(I_1,\ldots,I_k)}$ is a connected component of the inertia stack of \mb{1}{n}.
\end {proof}

\begin {notation} \label{doppianota} Let $\sigma \in S_k$. Then $\overline{Z}^{(I_1,\ldots,I_k)}=\overline{Z}^{(I_{\sigma(1),\ldots,\sigma(k)})}$. The twisted sector is identified up to isomorphism by $Z$ and the partition $\{I_1, \ldots I_k \}$ where the ordering of the $I_i$s does not matter. From now on we will simply denote this twisted sector in \mbun \ by $\overline{Z}^{\{I_1,\ldots,I_k\}}$ or $\overline{Z}^{\{I_1,\ldots,I_k\}}$: the set of parameters for the twisted sectors whose base space is $Z$ is the set of the $k$ partitions of $[n]$.
Note also that $\overline{Z}$ is identified with $\overline{Z}^{\{1\},\ldots,\{k\} }$ for every twisted sector $Z$ of \mb{1}{k}.
\end{notation}

With the notation just introduced, we state the main result of this section.

\begin {theorem} \label{twistedcompact} The decomposition of $I($\mbun$)$ in twisted sectors is (see Notation \ref{decompo}, \ref{notatio}, \ref{canonical}):
\begin{displaymath}\left( \overline{\mathcal{M}}_{1,n},1 \right) \sqcup \left(\overline{A_1}^{[n]}, -1\right) \sqcup \left(\overline{A_2}^{I_1,I_2}, -1\right) \space \sqcup\left(\overline{A_3}^{I_1,I_2,I_3}, -1\right) \end{displaymath} \begin{displaymath}  \sqcup \left(\overline{A_4}^{I_1,I_2,I_3,I_4}, -1\right) \sqcup  \left(C_4^{[n]}, i/-i\right) \sqcup \left(C_4^{I_1,I_2}, i/-i \right) 
\end{displaymath} \begin{displaymath}\sqcup  \left(C_6^{I_1,I_2}, \epsilon^2 / \epsilon^4\right)  \sqcup \left(C_6^{I_1,I_2,I_3}, \epsilon^2/\epsilon^4\right)
 \sqcup \left(C_6^{[n]}, \epsilon/ \epsilon^2 /\epsilon^4/ \epsilon^5\right),
\end{displaymath}
where each disjoint union is taken over the set of all possible decompositions of the set $[n]$ in $1,2,3$ or $4$ non-empty subsets: $[n]= \coprod I_i$.
\end {theorem}
\begin {proof} We have just seen in Theorem \ref{partialtwistedcompact} that all the components that appear in the decomposition are twisted sectors of \mb{1}{n}. We have to prove that there are no further connected components in the inertia stack of \mb{1}{n}. To see that there are no further twisted sectors, one can work by induction using the fact that
\begin{displaymath}
\pi_{[n]}: \overline{\mathcal{M}}_{1,n+1} \rightarrow \overline{\mathcal{M}}_{1,n} 
\end{displaymath}
is the universal curve.  
\end {proof}

\noindent From this, we obtain the following corollary that describes all the possible twisted sectors of \mbun \ stack-theoretically.
\begin {corollary} \label{fondamentale} Each twisted sector of $\overline{\mathcal{M}}_{1,n}$ is isomorphic to a product
\begin{displaymath}
A \times \overline{\mathcal{M}}_{0,n_1} \times \overline{\mathcal{M}}_{0,n_2} \times \overline{\mathcal{M}}_{0,n_3} \times \overline{\mathcal{M}}_{0,n_4},
\end{displaymath}
where $n_1,\ldots,n_4 \geq 3$ are integers and $A$ is in the set
\begin{displaymath}
\left\{B \mu_3, B \mu_4, B \mu_6, \mathbb{P}(4,6), \mathbb{P}(2,4), \mathbb{P}(2,2)\right\}.
\end{displaymath}

\end {corollary}
\begin{proof} It is a consequence of Theorem \ref{twistedcompact}, Lemma \ref{aicompact}, and the fact that $C_4 \cong B \mu_4$, $C_6 \cong B \mu_6$, $C_6'\cong C_6'' \cong B \mu_3$ (as a consequence of Theorem \ref{weierstrass}).
\end{proof}
\subsection{The cohomology of the inertia stack of \mbun}

We can use the results estabilished in Theorem \ref{twistedcompact} and Corollary \ref{fondamentale} to compute the dimension of the vector space $H^*_{CR}(\overline{\mathcal{M}}_{1,n}, \mathbb{Q})$ (see Definition \ref{crvector}). We write the formula for the dimension as a function of the dimension of $H^*(\overline{\mathcal{M}}_{0,n})$, which is well known after Keel \cite{keel}. Then, let
\begin{displaymath}h(n):= \dim H^*(\overline{\mathcal{M}}_{0,n+1}, \mathbb{Q})= \sum_k a^k(n)\end{displaymath} (the latter notation is the one of \cite[p. 550]{keel} shifted by $1$). 

\begin {corollary} \label{coomologiachenruan} The dimension of the Chen--Ruan cohomology vector space of \mbun \ is:
\begin{displaymath}\dim\left(H^*_{CR}\left(\overline{\mathcal{M}}_{1,n}, \mathbb{Q}\right)\right)= \dim (H^*(\overline{\mathcal{M}}_{1,n}, \mathbb{Q}))+ 8 h(n) + 3 \sum \ \binom{n}{i, j} \ h(i) h(j) +\end{displaymath} \begin{displaymath}+\frac{2}{3} \sum \ \binom{n}{i,j,k} \ h(i) h(j) h(k) + \frac{1}{12} \sum \ \binom{n}{i,j,k,l}\  h(i) h(j) h(k) h(l),\end{displaymath}
 where the sum is over indices $1 \leq i,j,k,l  \leq n$ such that their sum is $n$.
\end{corollary}

\begin{proof} This result is obtained from Theorem \ref{twistedcompact} and Corollary \ref{fondamentale}, using the fact that the dimension of the cohomology of a point is $1$ and the dimension of the cohomology of the projective line is $2$.
\end{proof}

We introduce the generating polynomials:

\begin{eqnarray}  \label{serietotale} P_0(s):=\sum_{n=0}^{\infty}\frac{Q_0(n)}{n!}s^n \\ P_1(s):=\sum_{n=0}^{\infty}\frac{Q_1(n)}{n!}s^n  \\ P_1^{CR}(s):=\sum_{n=0}^{\infty}\frac{Q_1^{CR}(n)}{n!}s^n
\end{eqnarray}

\noindent where:

\begin{eqnarray*}
Q_0(n):=\dim H^*(\overline{\mathcal{M}}_{0,n+1})=h(n) \\ 
Q_1(n):=\dim H^*(\overline{\mathcal{M}}_{1,n}) \\ 
Q_1^{CR}(n):=\dim H^*_{CR}(\overline{\mathcal{M}}_{1,n})
\end{eqnarray*}

\noindent with the convention that when the right hand side is not defined, the left hand side equals $1$. Formula \ref{coomologiachenruan} can now be written compactly.

\begin {theorem} The following equality between power series relates the dimensions of the cohomology group of $\overline{\mathcal{M}}_{0,n}$ and $\overline{\mathcal{M}}_{1,n}$ with the dimension of the Chen--Ruan cohomology group of $\overline{\mathcal{M}}_{1,n}$.
\begin {equation} \label {samuel}
 P_1^{CR}(s)=P_1(s)+8P_0(s)+3P_0(s)^2+\frac{2}{3}P_0(s)^3+\frac{1}{12}P_0(s)^4.
\end {equation} 
\end {theorem}

\subsection{The classes of the twisted sectors in \mbun}

\label{basedivisorisec}
In this section, we want to express the classes $[Y]$ for all $Y$ a twisted sector of \mbun, as linear combinations of elements in $R^*($\mbun$)$. It is possible to express them as linear combinations of products of divisor classes in \mbun. This is due to the fact that there are base twisted sectors (Definition \ref{base}) in genus $1$ only up to $4$ marked points, and Belorousski's Theorem \ref{belotheorem}. In fact, we manage to compute these classes as linear combinations of products of \emph{$S_n$-invariant} boundary divisor classes.

\begin{notation}\label{classi} If $Y$ is a base twisted sector (\ref{base}), we can write \begin{displaymath}[Y] \in A^*(\overline{\mathcal{M}}_{1,n})=R^*(\overline{\mathcal{M}}_{1,n})=H^{ev}(\overline{\mathcal{M}}_{1,n}):\end{displaymath} the two equalities hold when $n \leq 10$, see \ref{nminoredieci}. If $i:Y \to \overline{\mathcal{M}}_{1,n}$ is the restriction of the map from the inertia stack, $[Y]$ is the push--forward via $i$ of the fundamental class of the twisted sector $Y$ (see \ref{closedsubstack}).
\end{notation}

\noindent We use the notation for the divisors introduced in Notation \ref{divisorigenere0e1}. 

\begin {theorem} \label{basedivisori} Let $Y$ be a base twisted sector of $\overline{\mathcal{M}}_{1,n}$ (Definition \ref{base}). We express its class in the cohomology ring as a linear combination of products of $S_n$-invariant divisor classes.
\begin {itemize}
\item Base space classes coming from \mb{1}{1}:
\begin {enumerate} 
\item $[$\azerop$]$$=1$, the fundamental class of \mb{1}{1};
\item $[C_4]=\frac{1}{2} d_{irr}$
\item $[C_6]=\frac{1}{3} d_{irr}$.
\end {enumerate}
\item Base space classes coming from \mb{1}{2}:
\begin {enumerate} 
\item $[$\auno$]=\frac{1}{4} d_{irr}+3 d_{\{1,2\}} $;
\item $[C_4']=\frac{1}{2} d_{irr} d_{\{1,2\}}$;
\item $[C_6']=\frac{2}{3} d_{irr} d_{\{1,2\}}$.
\end{enumerate}
\item Base space classes coming from \mb{1}{3}:
\begin {enumerate} 
\item $[$\adue$]=\frac{1}{4} d_{irr} \left( \sum_{\{i,j\} \subset \{1,2,3\}} d_{\{i,j\}}\right)+ \frac{1}{4} d_{irr} d_{\{1,2,3\}}+$ \\ $\phantom{\overline{A}_3)=}+ 2 d_{\{1,2,3\}} \left(\sum_{\{i,j\} \subset \{1,2,3\}} d_{\{i,j\}}\right)$;
\item $[C_6'']=\frac{2}{9} d_{irr} d_{\{1,2,3\}}\left(\sum_{\{i,j\} \subset \{1,2,3\}} d_{\{i,j\}}\right) $.
\end {enumerate}
\item Base space classes coming from \mb{1}{4}:
\begin {enumerate} 
\item $[$\atre$]= \frac{1}{2} d_{\{1,2,3,4\}}\sum_{\{i,j,k\} \subset \{1,2,3,4\}} d_{\{i,j,k\}}\left( \sum_{\{l,m\} \subset \{i,j,k\}} d_{\{l,m\}}\right) +$ \\ $\phantom{\overline{A}_4=} + \frac{1}{12} d_{irr} \sum_{\{i,j,k\} \subset \{1,2,3,4\}} d_{\{i,j,k\}} \left( \sum_{\{l,m\} \subset \{i,j,k\}} d_{\{i,j\}}\right)+ $ \\$\phantom{\overline{A}_4=} + \frac{1}{12} d_{irr} d_{\{1,2,3,4\}}\left(\sum_{\{i,j\} \subset \{1,2,3,4\}} d_{\{i,j\}}\right) $.
\end {enumerate}
\end {itemize}
\end{theorem}
\begin {proof} We use the methods first established by Mumford \cite[Section III]{mumford}. For the classes of the points the result is trivial. 
We show how to obtain the result for the classes of the spaces $\overline{A_i}$. We refer to \cite{belo} for all the bases of the Chow groups of \mbun \ that we use in the following. 

First of all, a basis of $A^1($\mb{1}{2}$)$ is given by $d_{irr}$ and $d_{\{1,2\}}$. Therefore:
\begin {equation}\label{relazione}
[\overline{A_2}]=a d_{irr}+b d_{\{1,2\}}.
\end{equation}
Taking the push--forward via $\pi_{1}: \overline{\mathcal{M}}_{1,2} \to$ \mb{1}{1}, and using that the forgetful morphism restricted to \auno \ is of degree three (Lemma \ref{aicompact}), gives that $b=3$. Now we observe that $\overline{A_2}$ does not intersect $d_{\{1,2\}}$; by using that \begin{displaymath} d_{\{1,2\}} d_{\{1,2\}} = - \frac{1}{24}, \quad  d_{\{1,2\}} d_{irr}= 1/2, \end{displaymath} we obtain $a=\frac{1}{4}$.

A basis of the $S_3$-invariants of $A^2($\mb{1}{3}$)$ is given by:

\begin{displaymath}
d_{irr} \left( d_{\{1,2\}}+ d_{\{1,3\}}+ d_{\{2,3\}} \right), \ d_{irr} d_{\{1,2,3\}}, \ d_{\{1,2,3\}} \left(\sum_{\{i,j\} \subset \{1,2,3\}} d_{\{i,j\}}\right),
\end{displaymath}

\noindent therefore $[\overline{A_3}]$ can be uniquely written as:

\begin {displaymath}
a d_{irr} \left( d_{\{1,2\}}+  d_{\{1,3\}}+ d_{\{2,3\}} \right) +b  d_{irr} d_{\{1,2,3\}} +c d_{\{1,2,3\}} \sum_{\{i,j\} \subset \{1,2,3\}} d_{\{i,j\}}.
\end {displaymath}
Taking the push-forwards via the map forgetting one marked point, and using that these forgetful morphisms restricted to \adue \ are of degree two (Lemma \ref{aicompact}), gives:

\begin{displaymath}
2 a = \frac{1}{2}, \quad c=2.
\end{displaymath}
Now to determine $b$, intersect the class of \adue \ with $d_{\{1,2\}}$ to find $b=a$.

The class of $\overline{A_4}$ is computed similarly. The dimension of the $S_4$-invariants of $A^3(\overline{\mathcal{M}}_{1,4})$ is four; the four coefficients of the class of $\overline{A_4}$ are obtained by intersecting with the four $S_4$-invariant divisors of \mb{1}{4}. Another way to obtain three relations among the four coefficients is by considering the push-forward via the map that forgets the last marked point, this provides a nontrivial check of the result. 
\end {proof}

\begin {corollary} \label{corbasedivisori} Let $(Y,g)$ be a twisted sector of \mbun. Then $[Y]$ is in the subalgebra generated by the divisors of \mbun.
\end {corollary}
\begin {proof} As a consequence of Theorem \ref{twistedcompact}, every twisted sector class is $j_{k *} p^* ([Z])$, where $Z$ is a base twisted sector in \mb{1}{k}, $k \leq 4$, and the maps fit into the diagram:
\begin{displaymath}\xymatrix{ \overline{\mathcal{M}}_{1,k} \times \overline{\mathcal{M}}_{0,I_1+1} \times \ldots \times \overline{\mathcal{M}}_{0,I_k+1} \ar[rr]^{\hspace{2cm}j_k} \ar[d]^{p} && \overline{\mathcal{M}}_{1,n} \\
\overline{\mathcal{M}}_{1,k}, & \\
}
\end{displaymath}
where $j_{k}$ is the gluing map defined in Section \ref{section2b}, and $p$ is the projection onto the first factor. 

 Suppose without loss of generality that $1 \in I_1, \ldots, k \in I_k$, and $\pi_{\{1,\ldots,k\}}$ be the forgetful map. Then the class of $Y$ is just the transversal intersection of   \begin{displaymath} \pi_{\{1,\ldots, k\}}^*([Z]) \quad \textrm{and} \quad j_{k *}([\overline{\mathcal{M}}_{1,k}\times \overline{\mathcal{M}}_{0,I_1+1} \times \ldots \times \overline{\mathcal{M}}_{0,I_k+1}]).\end{displaymath} The second class is clearly in the subalgebra generated by the divisors, and so is the first since the class of $[Z]$ is in the subalgebra generated by the divisors by Theorem \ref{basedivisori}. 

\end {proof}

\section {The age grading}
\subsection {Definition of Chen--Ruan degree}
Let $X$ be a smooth algebraic stack of dimension $n$, and $x \in X$ a point. Let $T(X)$ be the tangent bundle of $X$. For any $g$ in the stabilizer group of $x$ of order $k$, there is a basis of $T_x(X)$ consisting of eigenvectors for the $g$-action. In terms of such a basis, the $g$-action is given by a diagonal matrix $M=$diag$(\xi^{a_1},\ldots,\xi^{a_n})$ where $\xi=e^{\frac{2 \pi i}{k}}$ and $a_i <k$. For any pair $(x,g)$, define a function $a(x,g):= \frac{1}{k} \sum a_i$. This function is nonnegative and takes rational values. Moreover, this function is constant on each connected component of the inertia stack (\cite[Chapter 3.2]{chenruan}).

\begin {definition} (\cite[Chapter 3.2]{chenruan}) Let $X$ be an algebraic smooth stack. The \emph{age} of a twisted sector $X_{(g)}$ is defined to be $a(x,g)$ for any point $(x,g) \in X_{(g)}$. 
\end {definition}

\noindent The age is also referred to as degree shifting number, or fermionic shift. The algebraic definition of age was given in \cite[7.1]{agv1} and \cite[7.1]{agv2}.

\begin {proposition} \label{codimension}Let $(Y,g)$, $(Y,g^{-1})$ be two connected components of the inertia stack of an algebraic stack $X$ which are exchanged by the involution of the inertia stack. Then
\begin{displaymath}
a((Y,g))+ a((Y,g^{-1}))= \codim_YX.
\end{displaymath}
\end {proposition}
\begin {remark} If $i:Y \rightarrow X$ is a twisted sector, and $x \in Y$ is a point, then the following splitting holds:\begin{displaymath}T_x X= T_x Y \oplus N_YX|_x.\end{displaymath} If $G:= \langle g \rangle$ is in the stabilizer group of $x$, then $T_xX$ is a representation of $G$ which splits as a sum of two representations: $T_x Y$ and $N_YX|_x$. The first of such representations is trivial by the definition of twisted sector. Therefore what is needed in order to compute the age of a twisted sector is to study the action of $G$ on $N_YX|_x$.
\end {remark}

In conclusion of this subsection we define the orbifold degree.

\begin{definition} We define the $d-th$ degree Chen--Ruan cohomology group as follows:
\begin{displaymath}
H^d_{CR}(X, \mathbb{Q}):= \bigoplus_i  H^{d-2 a(X_i,g_i)}(X_i, \mathbb{Q}),
\end{displaymath}
where the sum is over all the connected components of the inertia stack of $X$. If $\alpha$ is an element of this vector space, its \emph{orbifold degree} is $d$.
\end{definition}

\subsection {Age of \mun \ and \mbun}

 We start our computations in the smooth case. The result for the age of \m{1}{1} is well known, so we work out the other cases.

If $\phi_N$ is a generator of $\mu_N$ (and therefore of $\mu_N^{\vee}$ since we work over $\mathbb{C}$), we denote by $\langle \phi_N^k \rangle$ the one dimensional complex vector space with the action of $\mu_N$, where $\phi_N$ acts as the multiplication by $\phi_N^k$. Recall that if $G=\Aut(C,P)$ is the automorphism group of an elliptic curve, it is canonically identified with $\mu_N$ for a certain $N$ (Notation \ref{canonical}).
\label{sezioneeta}

\begin{lemma} \label{agemun} Let $Z \subset \mathcal{M}_{1,k}$ be the closed embedding of a twisted sector $(Z, \alpha)$ of $\mathcal{M}_{1,k}$, with $Z \in \{C_4, C_4', C_6, C_6', C_6''\}$. The normal bundle $N_Z \mathcal{M}_{1,k}$ is a representation of $\mu_N$ on a $k$-dimensional vector space:
\begin{itemize}
\item as a representation of $\mu_4$, $N_{C_4} \mathcal{M}_{1,1}$  is $\langle i^2 \rangle$,
\item as a representation of $\mu_4$, $N_{C_4'} \mathcal{M}_{1,2}$  is $\langle i^2 \rangle \oplus \langle i^3 \rangle$,
\item as a representation of $\mu_6$, $N_{C_6} \mathcal{M}_{1,1}$ is $\langle \epsilon^4 \rangle$,
\item as a representation of $\mu_3$ generated by $\epsilon^2$, $N_{C_6'} \mathcal{M}_{1,2}$  is $\langle \epsilon^2 \rangle \oplus \langle \epsilon^4 \rangle$,
\item as a representation of $\mu_3$ generated by $\epsilon^2$, $N_{C_6''} \mathcal{M}_{1,3}$  is $\langle \epsilon^2 \rangle \oplus \langle \epsilon^4 \rangle \oplus \langle \epsilon^4 \rangle$.
\end{itemize}
\end{lemma}

\begin {proof} The age for the twisted sectors of $\mathcal{M}_{1,1}$ is known from the description of it as an open substack of $\mathbb{P}(4,6)$ (see for instance \cite{mann}). It is easily checked that $i$ acts on the normal bundle as the multiplication by $i^6=i^2$ and $\epsilon$ acts as multiplication by $\epsilon^4$.

 We study the tangent bundle to $C_4'$ in \m{1}{2}, the other cases follow through. Since the forgetful morphism \mb{1}{2} $\to$ \mb{1}{1} is the universal curve, we have that the following diagram is cartesian (see \ref{c4c6} for the definition of $\mathcal{C}_4$):
\begin{displaymath}
\xymatrix{ C_4' \ar[r] \ar@/^0.4cm/^{g} [rr]& \mathcal{C}_4 \ar[r] & [ \mathcal{C}_4 / \mu_4 ] \ar[r] \ar[d] & \overline{\mathcal{M}}_{1,2} \ar[d] \\
&& C_4 \ar[r] & \overline{\mathcal{M}}_{1,1}.}
\end{displaymath}
 The normal bundle $N_{C_4'} \overline{\mathcal{M}}_{1,2}$ is therefore isomorphic as a representation of $\mu_4$ to the direct sum: $g^*(N_{[\mathcal{C}_4 / \mu_4]} \overline{\mathcal{M}}_{1,2} ) \oplus N_{C_4'} \mathcal{C}_4 $, since $\mathcal{C}_4 \to [\mathcal{C}_4 / \mu_4]$ is a finite \'etale map. The first term in the direct sum is equivariantly isomorphic to $N_{C_4} \overline{\mathcal{M}}_{1,1}$ (since the diagram is cartesian and the forgetful map is flat). The normal bundle $N_{C_4'} \mathcal{C}_4$ is equivariantly isomorphic to the tangent space to $\mathcal{C}_4$ in the second marked point. In the Weierstrass representation of $\mathcal{C}_4$ (\ref{c4c6}), this is the point with projective coordinates: $[0:0:1]$. The tangent space in these coordinates is then parametrized by $[0:t:1]$ and $i$ acts on it  as the multiplication by $i^3$ (Theorem \ref{weierstrass}).
\end {proof}
 
 \noindent With this result, it is straightforward to compute the age of all the twisted sectors of $\mathcal{M}_{1,n}$, by using the fact that the twisted sector corresponding to an involutive automorphism have age equal to half the codimension (Proposition \ref{codimension}). 

\begin{remark} \label{nomunozero} We observe that, in the proof of the above proposition, we had to use a different argument for the case $n=1$ and $n > 1$. This is due to the fact that there is no forgetful map from \m{1}{1}. In particular, if $(C,p) \in$ \m{1}{1}, and $\mu_N=\Aut(C,P)$ then the actions of $\mu_N$ on $N_{(C,p)}$ \m{1}{1} and on $T_p \mathcal{C}$ do not necessarily coincide.
\end{remark}
In the following Proposition, we give a description of the normal bundle $N_Z \overline{\mathcal{M}}_{1,n}$ where $Z$ is a twisted sector. We restrict to the cases where the base space of $Z$ (Definition \ref{base}) has dimension $0$, since the automorphism element of all others is involutive. Hence, let $Z$ be in $\{C_4, C_4', C_6,C_6', C_6''\}$. The normal bundle $N_Z \mathcal{M}_{1,k}$ is a $k$-dimensional vector space with an action of $\mu_N$ on it which we studied in Lemma \ref{agemun} ($N$ can be $3,4$ or $6$). For the sake of simplicity, we identify ${Z}^{I_1, \ldots, I_k}$ with $Z \times \overline{\mathcal{M}}_{0,I_1+1} \times \ldots \times \overline{\mathcal{M}}_{0,I_k+1}$ (see Definition \ref{base}). We denote by $p, p_1, \ldots, p_k$ the projections onto each factor.   We indicate by $\underline{\mathbb{C}}$ the trivial bundle of rank $1$.

\begin{proposition} \label{agembun} Let $I_1, \ldots, I_k$ be a partition of $[n]$ in $k \leq 3$ subsets. Suppose that $|I_i| >1$ for all $i$s. Then the normal bundle $N_{Z^{I_1,\ldots,I_k}} \overline{\mathcal{M}}_{1,n}$ is of rank $2k$ and splits as a direct sum of line bundles. \footnote{see last paragraph of Section \ref{sezionegerbe} for a line bundle on a trivial gerbe}
\begin{enumerate}
\item $N_{C_4^{[n]}} \overline{\mathcal{M}}_{1,n}$ is isomorphic to: $(i^2, \underline{\mathbb{C}}) \oplus (i^3, p_1^*(\mathbb{L}_{n+1}^{\vee}))$,
\item $N_{C_4^{I_1,I_2}} \overline{\mathcal{M}}_{1,n}$ is isomorphic to: $(i^2, \underline{\mathbb{C}}) \oplus(i^3, \underline{\mathbb{C}}) \oplus (i^3, p_1^*(\mathbb{L}_{I_1+1}^{\vee}))\oplus (i^3, p_2^*(\mathbb{L}_{I_2+1}^{\vee}))$,
\item $N_{C_6^{[n]}} \overline{\mathcal{M}}_{1,n}$ is isomorphic to: $(\epsilon^4, \underline{\mathbb{C}}) \oplus (\epsilon^5, p_1^*(\mathbb{L}_{n+1}^{\vee}))$,
\item $N_{C_6^{I_1,I_2}} \overline{\mathcal{M}}_{1,n}$ is isomorphic to: $(\epsilon^2, \underline{\mathbb{C}}) \oplus(\epsilon^4, \underline{\mathbb{C}}) \oplus (\epsilon^4, p_1^*(\mathbb{L}_{I_1+1}^{\vee}))\oplus (\epsilon^4, p_2^*(\mathbb{L}_{I_2+1}^{\vee}))$,
\item $N_{C_6^{I_1,I_2,I_3}} \overline{\mathcal{M}}_{1,n}$ is isomorphic to: $(\epsilon^2, \underline{\mathbb{C}}) \oplus(\epsilon^4, \underline{\mathbb{C}})\oplus(\epsilon^4, \underline{\mathbb{C}}) \oplus (\epsilon^4, p_1^*(\mathbb{L}_{I_1+1}^{\vee}))\oplus (\epsilon^4, p_2^*(\mathbb{L}_{I_2+1}^{\vee}))\oplus (\epsilon^4, p_3^*(\mathbb{L}_{I_3+1}^{\vee}))$.
\end{enumerate}
If some of the $I_i$s has cardinality $1$, the normal bundle has the same description after removing the corresponding components $p_i^*(\mathbb{L}_{I_i+1}^{\vee})$.
\end{proposition}

\noindent We postpone the proof of this proposition, in order to immediately see that as a consequence of it we can compute the age of all the twisted sectors of \mbun. We use the convention that $\delta(I)=\delta_{1,|I|}$, the Kronecker delta. 
\begin {corollary} 
\label{agembuncor} In this table we recollect the age, or degree shifting number, of the twisted sectors of \mbun.
\end {corollary}
\small
\ \\ 
\noindent
\begin {tabular}{|c|c|c|c|c|} \hline Component& Aut& Codimension& Age \\ \hline \hline
\azero & $-1$ &  $1- \delta(n)$ & $\frac{1}{2}(1- \delta(n))$ \\
\auno$^{I_1,I_2}$ & $-1$ &  $3- \delta(I_1) - \delta(I_2)$ & $\frac{1}{2}(3- \delta(I_1)- \delta(I_2) )$ \\
\adue$^{I_1,I_2,I_3}$ & $-1$ & $5- \delta(I_1) - \delta(I_2) -\delta(I_3)$ & $\frac{1}{2}(5- \delta(I_1)-\delta(I_2)- \delta(I_3) )$ \\
\atre$^{I_1,I_2,I_3,I_4}$ & $-1$ &  $7- \delta(I_1) - \delta(I_2) - \delta(I_3) -\delta(I_4)$ & $\frac{1}{2}(7- \delta(I_1)- \delta(I_2)-\delta(I_3)  -\delta(I_4) )$ \\
\hline
$C_4^{[n]}$ & $i$ & $2- \delta(n)$ & $\frac{5}{4} - \frac{3}{4} \delta(n)$ \\
$C_4^{[n]}$ & $-i$ & $2- \delta(n)$ & $\frac{3}{4} - \frac{1}{4} \delta(n)$ \\
$C_4^{I_1,I_2}$ & $i$ &  $4-\delta(I_1)-\delta(I_2)$ & $\frac{11}{4}- \frac{3}{4}(\delta(I_1) + \delta(I_2)) $\\

$C_4^{I_1,I_2}$ & $-i$ & $4-\delta(I_1)-\delta(I_2)$ & $\frac{5}{4}- \frac{1}{4}(\delta(I_1) +\delta(I_2)) $\\
\hline
$C_6^{I_1,I_2}$ &$\epsilon^2$&$4- \delta(I_1)-\delta(I_2)$ & $\frac{7}{3}- \frac{2}{3}(\delta(I_1)+\delta(I_2))$ \\

$C_6^{I_1,I_2}$ &$\epsilon^4$&$4- \delta(I_1)-\delta(I_2)$ & $\frac{5}{3}- \frac{1}{3}(\delta(I_1)+\delta(I_2))$ \\
$C_6^{I_1,I_2,I_3}$ & $ \epsilon^2$ &  $6- \delta(I_1)- \delta(I_2) - \delta(I_3)$ & $\frac{11}{3}- \frac{2}{3}(\delta(I_1)+\delta(I_2)+\delta(I_3))$ \\
$C_6^{I_1,I_2,I_3}$ & $ \epsilon^4$ &  $6- \delta(I_1)- \delta(I_2) - \delta(I_3)$ & $\frac{7}{3}- \frac{1}{3}(\delta(I_1)+\delta(I_2)+\delta(I_3))$ \\
\hline
$C_6^{[n]}$ &$\epsilon$& $2 - \delta(n)$& $\frac{3}{2}- \frac{5}{6} \delta(n)$ \\
$C_6^{[n]}$ &$\epsilon^2$& $2- \delta(n)$& $1 - \frac{2}{3} \delta(n)$ \\
$C_6^{[n]}$ &$\epsilon^4$& $2- \delta(n)$& $1- \frac{1}{3} \delta(n)$ \\
$C_6^{[n]}$ &$\epsilon^5$& $2- \delta(n)$& $\frac{1}{2} - \frac{1}{6} \delta(n)$ \\
\hline
\end {tabular}
\normalsize
\ \\
\begin {proof} The age of the sectors with an involutive automorphism is half the codimension by Lemma \ref{codimension}. They are all the twisted sectors whose associated base twisted sector (\ref{base}) has dimension $1$. For the remaining twisted sectors
, we simply apply Proposition \ref{agembun}.

\end {proof}

To prove Proposition \ref{agembun}, we use the following result due to Mumford:

\begin {lemma}  (\cite{mumford}) \label{referenzaimpossibile} Let $I_1, \ldots, I_k$ be a partition of $[n]$, and $j_k$ the gluing map defined in Section \ref{section2b}
\begin{displaymath}
j_k:\overline{\mathcal{M}}_{1, \coprod_{i=1}^k \bullet_i} \times \overline{\mathcal{M}}_{0,I_1 \sqcup \bullet_1} \times \ldots \times \overline{\mathcal{M}}_{0,I_k \sqcup \bullet_k} \to \overline{\mathcal{M}}_{1,n}.
\end{displaymath}
Let $p$ be the projection onto the first factor, and $p_1, \ldots, p_k$ the projections onto the moduli spaces of genus $0$ curves. Then the normal bundle of the map $j_k$ is isomorphic to:
\begin{displaymath}
N_{j_k}=\bigoplus_{1=1}^{k} p^*(\mathbb{L}_{\bullet_i}^{\vee}) \otimes p_i^*(\mathbb{L}_{\bullet_i}^{\vee}),
\end{displaymath}
where $\mathbb{L}_{\bullet_i}$ are the cotangent line bundles defined in Section \ref{section2b}, the first one is on $\overline{\mathcal{M}}_{1, \coprod_{i=1}^k \bullet_i}$, and the second one is on $\overline{\mathcal{M}}_{0,I_i \sqcup \bullet_i}$ 
\end {lemma}

\begin{proof} (of Proposition \ref{agembun})
If $j_k$ denotes as usual the gluing morphism defined in Section \ref{section2b}, the following diagram is cartesian by definition of $Z^{I_1, \ldots, I_k}$ ($p$ is the projection onto the first factor)
\begin{displaymath}
\xymatrix{{Z}^{I_1,\ldots,I_k} \ar[d]^{p} \ar@{^{(}->}[rr]^{\hspace{-2cm}\lambda} \ar@{}|{\square}[drr]&&  \mathcal{\overline{M}}_{1,k} \times \mathcal{\overline{M}}_{0,I_1 \sqcup \bullet_1} \times\ldots \times \mathcal{\overline{M}}_{0,I_k \sqcup \bullet_k} \ar[d]^{p} \ar@{^{(}->}[rr]^{\hspace{2cm}j_k} && \overline{\mathcal{M}}_{1,n} \\
{Z} \ar@{^{(}->}[rr] &&\mathcal{\overline{M}}_{1,k}. &&}
\end{displaymath}
In this case the following isomorphism holds
\begin{displaymath}p^*(N_{Z}\overline{\mathcal{M}}_{1,k}) \cong N_{Z^{I_1,\ldots,I_k}}  \mathcal{\overline{M}}_{1,k} \times \mathcal{\overline{M}}_{0,I_1 \sqcup \bullet_1} \times\ldots \times \mathcal{\overline{M}}_{0,I_k \sqcup \bullet_k} = N_{\lambda} \end{displaymath} since the diagram is cartesian and $p$ is flat. This is the trivial bundle with a certain representation of $\mu_N$, that we studied in Lemma \ref{agemun}. 

The map $j_k \circ \lambda$ is the restriction to $\overline{Z}^{I_1, \ldots, I_k}$ of the map $I(\overline{\mathcal{M}}_{1,n}) \to \overline{\mathcal{M}}_{1,n}$. Therefore, the normal bundle $N_{{Z}^{I_1, \ldots, I_k}} \overline{\mathcal{M}}_{1,n}$ is isomorphic to the direct sum: $\lambda^* N_{j_k} \oplus N_{\lambda}$. To conclude the proof, we have now to study the first term $\lambda^* N_{j_k}$.

According to Lemma \ref{referenzaimpossibile}, we have:
\begin{displaymath}
\lambda^*N_{j_k}= \left(\bigoplus_{1=1}^{k} p^*(\mathbb{L}_{\bullet_i}^{\vee}) \otimes p_i^*(\mathbb{L}_{\bullet_i}^{\vee}) \right)_{|Z^{I_1, \ldots, I_k}}.
\end{displaymath}
The term $ p^*(\mathbb{L}_{\bullet_i}^{\vee})_{|Z^{I_1, \ldots, I_k}}$ is the trivial bundle whose constant fiber is the tangent space to $Z$ in the $i$-th marked point. It carries the representation of $\mu_N$, as the latter group acts on the tangent space to $Z$ in the $i$-th marked point (this action can be computed explicitly, see Lemma \ref{agemun}). On the other hand, the term $p_i^*(\mathbb{L}_{\bullet_i}^{\vee})_{|Z^{I_1, \ldots, I_k}}$ carries the trivial representation of $\mu_N$, but it is non trivial as a line bundle. This last observation concludes the proof.
\end{proof}

We can now write a formula analogous to \eqref{samuel}, adding a new variable to separate the different degrees in the Chen--Ruan cohomology of \mb{1}{n}. We define:
\begin{eqnarray} \label{seriepotenze} P_0(s,t):=\sum_{n=0}^{\infty}\frac{Q_0(n,m)}{n!} s^n t^m, \\ P_1(s,t):=\sum_{n=0}^{\infty}\frac{Q_1(n,m)}{n!}s^n t^m , \\ 
P_{1, \alpha}^{CR}(s,t):=\sum_{n=0}^{\infty}\frac{Q_{1, \alpha}^{CR}(n,m)}{n!}s^n t^m,
\end{eqnarray}
\noindent where
\begin{eqnarray*} 
Q_0(n,m):=\dim H^{2m}(\overline{\mathcal{M}}_{0,n+1})=a^m(n), \\ 
Q_1(n,m):=\dim H^{2m}(\overline{\mathcal{M}}_{1,n}), \\ 
Q_{1, \alpha}^{CR}(n):=\dim H^{2m+ \alpha}_{CR}(\overline{\mathcal{M}}_{1,n}).
\end{eqnarray*}

\noindent The relevant values of $\alpha$ are $\{0, \frac{1}{4}, \frac{1}{3}, \frac{1}{2}, \frac{2}{3}, \frac{3}{4}\}=: A$. Summing over $\alpha$ and shifting by $\alpha$ the degree of $t$, we have the generating series of the orbifold Poincar\'e polynomials:\footnote{Note that the orbifold Poincar\'e polynomials are, strictly speaking, not polynomials.} 
\begin{displaymath}
P_{1}^{CR}(s,t):= \sum_{\alpha \in A} t^{\alpha} P_{1, \alpha}^CR (s,t).
\end{displaymath} 
In other words, the coefficient of degree $n$ in the variable $s$ of $P_{1}^{CR}$ is the orbifold Poincar\'e polynomial of $\overline{\mathcal{M}}_{1,n}$ divided by $n!$. The power series \eqref{seriepotenze} are described in \cite[Theorem 5.9]{getzleroperads}, \cite[Theorem 2.6]{getzler2}. There, the author describes the cohomology of the moduli of genus $0$ and genus $1$, $n$-pointed stable curves as an $S_n$-representation. 

\begin {theorem} \label{poincare1} The following equalities between power series relate the dimension of the $m$-th Chen--Ruan cohomology group of $\overline{\mathcal{M}}_{1,n}$ ($n>4$) with the Betti numbers of the moduli of pointed stable curves of genus $0$ and $1$. 
\begin{eqnarray*} 
P_{1, 0}^{CR}(s,t)&=&P_1 + (t + t^2) P_0+ 3 s (t^2+t^3) P_0^2+ \frac{s}{24} (t^3+ t^4) P_0^3+ \\ & & 2 s (t+t^2) \frac{\partial}{\partial s} (s P_0) +  \frac{s^3}{6} (t^2+t^3) \frac{\partial^3}{\partial s^3}(s^3 P_0)\\
P_{1, \frac{1}{4}}^{CR}(s,t)&=& t P_0 + \frac{t^2}{2} P_0^2\\
P_{1, \frac{1}{3}}^{CR}(s,t)&=&\frac{t^2}{2}P_0^2+\frac{t^2}{6} P_0^3+ t s \frac{\partial}{\partial s} (s P_0)+ \frac{s^2 t^2}{2} \frac{\partial^2}{\partial s^2}(s^2 P_0)\\
P_{1, \frac{1}{2}}^{CR}(s,t)&=&2 (1+t) P_0+ \frac{t+t^2}{2} P_0^2 + \frac{t^2+t^3}{6} P_0^3+ \frac{t^3+t^4}{24}P_0^4+ \\ & & +\frac{s^2}{2} (t+t^2) \frac{\partial^2}{\partial s^2}(s^2 P_0) \\
P_{1, \frac{2}{3}}^{CR}(s,t)&=&\frac{t}{2} P_0^2 + \frac{t^3}{6} P_0^3 + s t \frac{\partial}{\partial s} (s P_0)+ \frac{s^2}{2} t \frac{\partial^2}{\partial s ^2} (s^2 P_0)\\
P_{1, \frac{3}{4}}^{CR}(s,t)&=&  P_0 + \frac{t}{2} P_0^2\\
\end{eqnarray*}
\end{theorem}

\section {The second inertia stack}

The definition of the Chen--Ruan product involves the second inertia
stack.

\begin {definition} Let $X$ be an algebraic stack. The \emph{second
inertia stack} $I_2(X)$ is defined as:
\begin{displaymath} I_2(X)=I(X) \times_X I(X).\end{displaymath}
\end {definition}

\begin{remark} \label{liscezza2} Like the inertia stack, the second inertia stack is smooth (cfr. \ref{liscezza1}, see also \cite[p. 15]{agv1}, after noticing that $\mathcal{K}_{0,3}(X,0) \cong I_2(X)$).
\end{remark}

\noindent A point in $I_2(X)$ is a triplet
$(x,g,h)$ where $x$ is a point of $X$ and $g, h \in$ Aut$(x)$. It can equivalently be given as $(x,g,h, (g h)^{-1})$.

\begin {remark} \label{doppinerzia} The second inertia stack comes with three natural morphisms to
$I(X)$: $p_1$ and $p_2$, the two projections of the fiber product,
and $p_3$ which acts on points sending $(x,g,h)$ to $(x,g h)$.
This gives the following diagram, where $(Y,g,h,(gh)^{-1})$ is a double twisted sector and $(X_1,g)$, $(X_2,h)$, $(X_3, (gh))$ are twisted sectors:
\begin {displaymath}  \xymatrix{ & (X_1,g)\\
(Y,g,h) \ar@/^/[ur]^{p_1} \ar[r]^{p_2} \ar@/_/[dr]^{p_3}& (X_2,h)\\
& (X_3,g h).\\}
\end {displaymath}
\end {remark}

Let us now study the double twisted sectors in the case when $X=$\mbun. From now on we focus on the compact case, since the case of \mun \ follows through analogously and much more simply. The following follows easily from the fact that all isotropy groups of \mbun \ are cyclic (cf. also Remark \ref{closedsubstack}).

\begin {proposition} \label{doublesingle} Let $(Z,g,h,(gh)^{-1})$ be a double twisted sector of \mbun. Then either $(Z,g)$ or $(Z,h)$ or $(Z, (gh)^{-1})$ is a twisted sector of the inertia stack of \mbun.
\end {proposition}

\begin {remark} \label{lista0} We label each sector of $I_2(X)$ via the triplet
$(g,h,(g h)^{-1})$. There are two automorphism groups acting on
$I_2(X)$: an involution sending a sector labeled with $(g,h, (g
h)^{-1})$ into $(g^{-1},h^{-1},(g h))$, and $S_3$ which permutes the
three automorphisms. Up to permutations and involution, the
following are all the possible labels of the sectors in $I_2(X)$ that correspond to non-empty connected sectors:
\begin{itemize}
\item $(1,1,1)$, generated group $\mu_1$; \item $(1,-1,-1)$,
generated group $\mu_2$; \item
$(\epsilon^2,\epsilon^2,\epsilon^2)$, generated group
$\mu_3$;\item $(1,\epsilon^2,\epsilon^4)$, generated group
$\mu_3$; \item $(1,i,-i)$ generated group $\mu_4$; \item
$(i,i,-1)$, generated group $\mu_4$; \item
$(\epsilon,\epsilon,\epsilon^4)$, generated group $\mu_6$;\item
$(\epsilon,\epsilon^2,-1)$, generated group $\mu_6$;\item
$(1,\epsilon,\epsilon^5)$, generated group $\mu_6$.
\end{itemize}

\end {remark}

\noindent We now describe the sectors of the double inertia stack.
We do so up to the automorphisms described in the previous remark,
and up to the permutations of the marked points. The next proposition is easily obtained after Theorem \ref{twistedcompact}.

\begin{proposition} \label{doubletwisted} Up to permutation of
the automorphisms, and up to involution, the following are the twisted sectors of
$I_2(\overline{\mathcal{M}}_{1,n})$:

\begin{displaymath}
\left(\textrm{\azero}, (1,-1,-1)\right),\quad
\left(\textrm{\auno}^{I_1,I_2}, (1,-1,-1)\right), \quad
\left(\textrm{\adue}^{I_1,I_2,I_3}, (1,-1,-1)\right)\end{displaymath} \begin{displaymath}
\left(\textrm{\atre}^{I_1,I_2,I_3,I_4}, (1,-1,-1)\right), \quad
 \left(C_6^{I_1,I_2}, (1, \epsilon^2,
\epsilon^4)/(\epsilon^2,\epsilon^2,\epsilon^2)\right) \end{displaymath} \begin{displaymath} \left(C_6^{I_1,I_2,I_3}, (1, \epsilon^2,
\epsilon^4)/(\epsilon^2,\epsilon^2,\epsilon^2)\right), \quad
\left(C_4^{[n]}, (1,i,-i)/(i,i,-1)\right) \end{displaymath} \begin{displaymath}
\left(C_4^{I_1,I_2}, (1,i,-i)/(i,i,-1)\right),
 \quad \left(C_6^{[n]},
(1,\epsilon,\epsilon^5)/(\epsilon,\epsilon,\epsilon^4)/(\epsilon,\epsilon^2,-1)/(\epsilon^2,\epsilon^2,\epsilon^2)\right),
\end{displaymath}
where $I_1,I_2,I_3,I_4$ form a partition of $[n]$ in non-empty subsets.
\end {proposition}

\section {The excess intersection bundle}
\subsection{Definition of the Chen--Ruan product}
We review the definition of the excess intersection bundle over
$I_2(X)$, for $X$ an algebraic smooth stack. Let $(Y,g,h, (g h)^{-1})$ be a twisted sector of
$I_2(X)$. Let $H:= \langle g, h \rangle$ be the group
generated by $g$ and $h$. 

\begin{construction} \label{costruzione} \emph{
Let $\gamma_0, \gamma_1, \gamma_{\infty}$ be three small loops around $0, 1, \infty \subset \mathbb{P}^1$. Any map $\pi_1(\mathbb{P}^1 \setminus \{0,1, \infty\}) \to H$ corresponds to an $H$-principal bundle on $\mathbb{P}^1 \setminus \{0,1, \infty\}$.
Let $\pi^0 : C^0 \rightarrow \mathbb{P}^1 \setminus \{0, 1 , \infty\}$
be the $H$-principal bundle which corresponds to the map $\gamma_0 \to g, \gamma_1 \to h, \gamma_{\infty} \to (gh)^{-1}$. It can be uniquely extended to a ramified $H$-Galois covering $C \to \mathbb{P}^1$ (see \cite[Appendix]{fantechigottsche}), where $C$ is a smooth compact curve. Note that $H$ acts on $C$ (and, by definition, $\mathbb{P}^1$ is the quotient $C/H$ as varieties), and hence on $H^1(C, \mathcal{O}_C)$.
}\end{construction}

Let $f:Y \to X$ be the restriction of the canonical map $I_2(X) \to X$ to the twisted sector $Y$. Note that $H$ acts on $f^*(T_X)$.

\begin{definition} \cite{chenruan} \label{eccesso} With the same notation as in the previous
paragraph, the \emph{excess intersection bundle} over $Y$ is defined as:
\begin{displaymath}
E_Y= \left(H^1(C, \mathcal{O}_C) \otimes_{\mathbb{C}} f^*(T_X) \right)^H,
\end{displaymath}
\emph{i.e.} the $H$-invariant subbundle of the expression between parenthesis.
\end{definition}

\begin{remark} Since $H^1(C, \mathcal{O}_C)^H=0$, it is the same
to consider in the previous definition:
\begin{displaymath}
\left(H^1(C, \mathcal{O}_C) \otimes N_YX \right)^H,
\end{displaymath}
where $N_Y$ is the coker of $T_Y \to f^*(T_X)$. 
\end{remark}
 We now review the definition of the Chen--Ruan product.

\begin{definition} \label{orbprod} Let $\alpha \in H^*_{CR}(X)$, $\beta \in
H^*_{CR}(X)$. We define:
\begin{displaymath}
\alpha *_{CR} \beta =p_{3 *} \left( p_1^*(\alpha) \cup
p_2^*(\beta) \cup c_{top}(E) \right).
\end{displaymath}

\end{definition}

\begin {theorem}\label{prodotto} (\cite{chenruan}) With the age grading defined in the previous section, $(H^*_{CR}(X, \mathbb{Q}), *_{CR})$ is a graded $(H^*(X,\mathbb{Q}), \cup)$-algebra. 
\end {theorem}

Theorem \ref{prodotto} allows us to compute the rank of the excess intersection bundle in terms of the already computed age grading. If $(Y,(g,h,(gh)^{-1}))$ is a sector of the second inertia stack, the rank of the excess intersection bundle is (here we stick to the notation introduced in Remark \ref{doppinerzia}):
\begin {equation}\label{formulaeccesso1}
\textrm{rk}(E_{(Y,g,h)})= a(X_1,g)+a(X_2,h)+a(X_3,(gh)^{-1})- \codim_YX.
\end {equation}

\begin {corollary}\label{semplifica} The excess intersection bundle over double twisted sectors when either $g$,$h$, or $(gh)^{-1}$ is the identity, is the zero bundle.
\end {corollary}

One other useful consequence of Proposition \ref{codimension} and Theorem \ref{prodotto} relates the rank of the excess bundle over a double twisted sector and the rank of the excess bundle over the double twisted sector obtained inverting the automorphisms that label the sector
\begin {equation} \label{formulaeccesso2}
\rk(E_{(Y,g^{-1},h^{-1})})= \sum_{i=1}^3 \codim_{X_i}X-2 \codim_YX - \rk(E_{(Y,g,h)}).
\end {equation}

\noindent Let $(Y,(g,h,(gh)^{-1}))$ be a double twisted sector in $I_2($\mbun$)$, and let $H$ be the group generated by $(g,h,(gh)^{-1})$. We will study $N_YX$ and $H^1(C, \mathcal{O}_C)$ as $H$-representations. 

\subsection{The excess intersection bundle for \mbun}

We have seen in Remark \ref{lista0} what are all the possible couples of automorphisms that correspond to non-empty connected substacks of $I_2(\overline{\mathcal{M}}_{1,n})$. Thanks to Corollary \ref{doublesingle}, the double twisted sectors whose excess intersection bundles have non-zero rank are those labelled by:
\begin{equation} \label{lista}
(\epsilon^2,\epsilon^2,\epsilon^2), (i,i,-1), (\epsilon,\epsilon, \epsilon^4), (\epsilon,\epsilon^2,-1)
\end{equation}
up to permutation and involution.
The top Chern classes of the excess intersection bundles for \mun \ are always $0$ or $1$, since the coarse moduli spaces of the double twisted sectors labeled by these automorphisms are points.

The ranks of the excess intersection bundles for the twisted sectors labeled by \ref{lista} can be computed thanks to formulas \eqref{formulaeccesso1} and \eqref{formulaeccesso2}.
\begin {proposition} \label{lista2}In the following table we list the ranks of the excess intersection bundles over all the double twisted sectors $(Z,g,h)$ of \mbun, such that $g,h$ and $gh \neq 1$:
\ \\
\center \begin {tabular}{|c|c|c||c|c|}
\hline
 $(g,h)$ & \emph{Double twisted sector}& \emph{rk}$(E)$&$(g^{-1},h^{-1})$  & \emph{rk}$(E)$\\
\hline

$(\epsilon^2,\epsilon^2)$ & $C_6^{[n]}$ & $1$ & $(\epsilon^4, \epsilon^4)$ & $1$ \\
$(\epsilon^2,\epsilon^2)$ & $C_6^{I_1,I_2}$ & $3$ & $(\epsilon^4, \epsilon^4)$ & $1$ \\
$(\epsilon^2,\epsilon^2)$ & $C_6^{I_1,I_2,I_3}$ & $5$ & $(\epsilon^4, \epsilon^4)$ & $1$ \\
\hline

$(i,i)$ & $C_4^{[n]}$ & $1$ & $(-i,-i)$ & $0$ \\

$(i,i)$ & $C_4^{I_1,I_2}$ & $3$ & $(-i,-i)$ & $0$ \\
\hline
$(\epsilon,\epsilon)$ & $C_6^{[n]}$ & $2$ & $(\epsilon^5, \epsilon^5)$ & $0$ \\
\hline
$(\epsilon,\epsilon^2)$ & $C_6^{[n]}$ & $1$ & $(\epsilon^4, \epsilon^5)$ & $0$ \\
\hline

\end {tabular}
\end {proposition}
\ \\
\begin {remark} \label{semplifica2} The above proposition and Corollary \ref{semplifica} imply that many top Chern classes of excess intersection bundles are $1$ (those of the  rank $0$ excess intersection bundles). Moreover it is obvious, after Proposition \ref{agembun}, that all the rank $3$ or $5$ bundles of Corollary \ref{semplifica} contain at least one trivial subbundle of rank $1$, this implying that their top Chern classes is zero.
\end {remark}

Now we want to compute explicitly the remaining, non trivial, excess intersection bundles and their top Chern classes for \mbun. We study the decomposition of $H^1(C, \mathcal{O}_C)$ as a representation of $H$ in the cases corresponding to non-zero ranks in Proposition \ref{lista2}. Here $\phi_N$ is a generator for $\mu_N^{\vee}$, the same chosen in Section \ref{sezioneeta}. As in that section, we indicate by $\langle \phi_N^k \rangle$ the one dimensional complex vector space endowed with the action of $\phi_N \in \mu_N^{\vee}$ given by the product by $\phi_N^k$.

\begin {proposition} \label{h} Let $H$ be generated by two elements $g,h$ as in \ref{lista2}. Let $C \to \mathbb{P}^1$ be the $H$-covering associated with the generators $g,h$ (see Construction \ref{costruzione}). We study $H^1(C, \mathcal{O}_C)$ as an $H$-representation: \begin {itemize}
\item $H= \mu_3$: $g= \epsilon^2, h=\epsilon^2$, then $H^1(C, \mathcal{O}_C)= \langle \epsilon^2 \rangle$,
\item $H= \mu_3$: $g= \epsilon^4, h=\epsilon^4$, then $H^1(C, \mathcal{O}_C)= \langle \epsilon^4 \rangle$,
\item $H=\mu_4$: $g=i,h= i$, then $H^1(C, \mathcal{O}_C)= \langle i \rangle$,
\item $H= \mu_6$: $g=\epsilon, h=\epsilon^2$, then $H^1(C, \mathcal{O}_C)= \langle \epsilon \rangle$,
\item $H=\mu_6$: $g=\epsilon, h=\epsilon$, then $H^1(C, \mathcal{O}_C)= \langle \epsilon \rangle \oplus \langle \epsilon^2 \rangle$.
\end {itemize}
\end {proposition}
\begin {proof} By Serre duality, we compute the action of $H$ on $H^0(C, \Omega^1_C)$. In the first four cases, the curve $C$ has genus $1$, thus $\Omega^1_C$ is trivial, therefore the action of $H$ coincides with its action on the cotangent space of a point of total ramification. In the fifth case the curve $C$ has genus $2$; by choosing suitable coordinates one can represent respectively $C$ and a basis for $H^0(C, \Omega^1_C)$ as
\begin{displaymath}
y^2=x^6-1, \quad \left(\frac{d x}{y}, \ x \frac{d x}{y}\right).
\end{displaymath}
Here the automorphism acts by fixing $y$ and mapping $x \to \epsilon x$. From this description the result for the last case follows.
\end {proof}

\noindent With all this, and thanks to Proposition \ref{agembun}, we can compute the excess intersection bundles and their respective top Chern classes. We know already that, among the list of couples of automorphisms of Proposition \ref{lista2}, the rank $3$ and $5$ bundles have top Chern class zero (see \ref{semplifica2}). Among the vector bundles having rank greater than zero, we can prove:

\begin {corollary} \label{zero} In Table \ref{lista2}, the top Chern classes of all the excess intersection bundles (which are all line bundles) corresponding to the couple $(\epsilon^4, \epsilon^4)$ are zero. The top Chern class of the excess intersection bundle that corresponds to the couple $(\epsilon, \epsilon)$ (of rank $1$), is also zero.
\end {corollary}
\begin {proof} From Proposition \ref{agembun} and Proposition \ref{h}, it is straightforward to see that all the excess bundles mentioned in the statement contain a trivial subbundle, forcing their top Chern class to be zero.
\end {proof}

So we are now left with three rank $1$ excess intersection bundles, whose top Chern class in non-zero and not $1$.
In the following diagram and in the following lemma, we identify the isomorphic spaces in order to simplify the notation for the projection maps.
\begin{displaymath}
\xymatrix{&& B \mu_{\lambda} \\
C_{\lambda}^{[n]} \ar[r] \ar@/^/[urr]^{p} \ar@/_/[drr]_{p_1}& B \mu_{\lambda} \times \overline{\mathcal{M}}_{0,n\sqcup \bullet} \ar[ur]^{p} \ar[dr]_{p_1}& \\
&& \overline{\mathcal{M}}_{0,n\sqcup \bullet}},
\end{displaymath}
where $\lambda$ can be $4$ or $6$. Here $\bullet$ is the gluing point.

\begin {corollary} \label{topchern} The only top Chern classes of the excess intersection bundles over double twisted sectors of \mbun \ that are not $0$ nor $1$ are:
\begin {enumerate}
\item $(C_6^{[n]}, (\epsilon^2,\epsilon^2,\epsilon^2)) \cong B_{\mu_6} \times \overline{\mathcal{M}}_{0,n\sqcup \bullet}$, where the top Chern class of the excess intersection bundle is $-p_1^*(\psi_{\bullet})=-p_1^*(\psi_{n+1})$;
\item $(C_4^{[n]}, (i,i,-1))  \cong B_{\mu_4} \times \overline{\mathcal{M}}_{0,n\sqcup \bullet}$, where the top Chern class of the excess intersection bundle is  $-p_1^*(\psi_{\bullet})=-p_1^*(\psi_{n+1})$;
\item $(C_6^{[n]}, (\epsilon, \epsilon^2, -1)) \cong B_{\mu_6} \times \overline{\mathcal{M}}_{0,n\sqcup \bullet}$, where the top Chern class of the excess intersection bundle is $-p_1^*(\psi_{\bullet})=-p_1^*(\psi_{n+1})$.
\end {enumerate}
\end {corollary}
\begin{proof} The fact that all other top Chern classes are zero or one follows from all the considerations in this section. In particular, we have observed in the beginning of the section that the excess intersection bundle that may have top Chern class different from $1$ are listed in \ref{lista2}. In Remark \ref{semplifica2} and in Corollary \ref{zero} we have computed the top Chern class of all the remaining cases to be zero or $1$.

The result stated then follows as a consequence of Proposition \ref{agembun}, Proposition \ref{h} (and from the very definition of the $\psi$ classes, see Section \ref{section2b}).
\end{proof}

\noindent Note that when $n=2$ the top Chern classes in Corollary above are $0$ too, because the sectors involved are all points.

To conclude, we summarize the result we have obtained in this section:
\begin {theorem} \label{topchernt} All top Chern classes of the excess intersection bundles over all double twisted sectors are explicitly given. They can be:
\begin {enumerate}
\item either $1$, for all the sectors listed in Proposition \ref{doubletwisted} such that one of the three automorphisms of the labeling is $1$,
\item or again $1$, for some of the sectors in the list \ref{lista2}, more precisely those mentioned in Remark \ref{semplifica2},
\item or $0$, for some of the sectors listed in \ref{lista2}, as discussed in Remark \ref{semplifica2} and in Corollary \ref{zero};
\item or a pullback of a $\psi$ class over a component \mb{0}{n}, for the the remaining elements of the list \ref{lista2}, as in Corollary \ref{topchern}.
\end {enumerate}
\end {theorem}

\section {Pull--backs and push--forwards of strata to the twisted sectors}
\label{sezionepull}
In order to compute the Chen--Ruan product, one has to compute pull--backs from the twisted sectors to the double twisted sectors and push--forwards from the double twisted sectors to the twisted sectors. Thanks to Corollary \ref{semplifica}, it is necessary and sufficient to only compute push--forwards and pull--backs between twisted sectors of the inertia stack.

In this section we fix $n$ and call $X:=$\mb{1}{n}. Let $(Y,g)$ be a twisted sector of $X$, and $f:Y \to X$ be the closed embedding of the twisted sector.

\begin {lemma}  \label{cycle} The cycle map
\begin{displaymath}
A^*(Y, \mathbb{Q}) \rightarrow H^{2*}(Y, \mathbb{Q})
\end{displaymath}
\noindent is a graded ring isomorphism. Moreover the Chow ring of each twisted sector is generated by divisors. 
\end {lemma}
\begin {proof} All factors of each twisted sector have Chow ring isomorphic to the even cohomology. The cohomology ring of \mb{0}{n} is generated by divisors due to the work of Keel \cite{keel}. Each of the spaces \azerop, \auno, \adue, \atre \space has coarse moduli space isomorphic to $\mathbb{P}^1$.
\end {proof}

We can now state and prove the result announced in the introduction. For some of the results needed in the proof we refer to the following two subsections on pull--backs and push--forwards.

\begin {theorem} \label{fundamental} The Chen--Ruan cohomology ring of \mbun \ is generated as an $H^*(\overline{\mathcal{M}}_{1,n}, \mathbb{Q})$-algebra by the fundamental classes of the twisted sectors with explicit relations. 
\end {theorem}
\begin {proof}  We will prove in Corollary \ref{surge} that $f^*$, the induced pull--back in cohomology, is surjective. This suffices to prove the generation part of the statement. Indeed, let $(X_i,g_i)$ be two twisted sectors, and $f_i: (X_i,g_i) \to X$ be the canonical projection maps. Let $\alpha_i \in H^*((X_i,g_i), \mathbb{Q})$ for $i=1,2$.  Let $\tilde{\alpha_i}$ be two liftings of $\alpha_i$ to $H^*(\overline{\mathcal{M}}_{1,n})$ obtained by using the surjectivity of $f_i^*$. Then we have 
 \begin{displaymath}
\alpha_1 *_{CR} \alpha_2 = (\tilde{\alpha_1} \cup \tilde{\alpha_2}) *_{CR} 1_{(X_1,g)} *_{CR} 1_{(X_2,h)},
\end{displaymath}
 this proves the generation claim.
 
  Then, as a consequence of Proposition \ref{doublesingle}, $(X_1,g) \times_X (X_2,h)$ is connected, and hence a double twisted sector of $I_2(X)$. Let $(X_3,g_1 g_2)$ be such a twisted sector of $I(X)$, and denote by $f_3$ the closed embedding of $(X_3, g_1 g_2)$ in $X$.  Let $p_3: (X_1,g_1) \times_X (X_2,g_2) \rightarrow (X_3, g_1 g_2)$ be the third projection of the double twisted sector as in  \ref{doppinerzia}. Let $E$ be the excess intersection bundle on $(X_1,g_1) \times_X (X_2,g_2)$, and $\gamma:= p_{3 *}(c_{top}(E))$. 

In \ref{choicecorollary} we fix a candidate, for every couple $X_1$, $X_2$ of twisted sectors, of a cohomology class $\beta=\beta\left((X_1,g_1),(X_2,g_2)\right) \in H^*(X,\mathbb{Q})$ such that $p_{3*}(c_{top}(E))= f_3^*(\beta)$. In Section \ref{pullsection} we fix a set of multiplicative generators for $H^*(X, \mathbb{Q})$: boundary divisors, subbanana cycles (\ref{bananacycle}) and any fixed set of additive generators for the odd cohomology of $X$. Our relations depend upon these choices. Finally, we obtain the formula for the Chen--Ruan product

\begin{equation} \label{relazioni}
\alpha_1 *_{CR} \alpha_2= f_3^*(\tilde{\alpha_1} \cup \tilde{\alpha_2} \cup \beta)=  \tilde{\alpha}\left( (X_1,g_1), (X_2,g_2) \right) *_{CR} 1 _{(X_3, g_1 g_2)},  
\end{equation}
where we posed $\tilde{\alpha}\left((X_1,g), (X_2,g_2)\right):=\tilde{\alpha_1} \cup \tilde{\alpha_2} \cup \beta\left((X_1,g_1),(X_2,g_2)\right)$.

 From \eqref{relazioni}, the relations appear naturally divided into two sets.  The first set of relations, presented in Section \ref{prodfund} comes from the Chen--Ruan product of all couples of twisted sectors. The second set of relations comes, for each twisted sector $Y$, from the classes of \mbun \ that are in the kernel of $f^*$. These are all the relations of $H^*(\overline{\mathcal{M}}_{1,n})$ \emph{as a module} over $H^*_{CR}(\overline{\mathcal{M}}_{1,n})$ generated by the fundamental classes of the twisted sectors.  These relations are determined in Section \ref{pullsection}. 
 \end {proof}

\noindent
Note that we actually obtain finitely many generators of the \emph{even part} of $H^*_{CR}(\overline{\mathcal{M}}_{1,n})$ as a $\mathbb{Q}$-algebra as a consequence of Section \ref{pullsection}. What we mean by ``even part'' here is in the original grading on $I($\mbun$)$, without considering the age grading shift. The Chen--Ruan cohomology of \mbun \ as a $\mathbb{Q}$-algebra is generated by the fundamental classes of the twisted sectors, the boundary divisors of \mbun \ and the subbanana cycles (see Definition \ref{bananacycle}). In this description the relations are the two sets of relations described in the proof of Theorem \ref{fundamental}, and the set of relations for $R^*($\mbun$)$ (see \cite{petersen}).

We observe that the first part of the proof of Theorem \ref{fundamental} shows that the fundamental classes of the twisted sectors generate the stringy Chow ring of \mbun \ (see \cite{agv1}, \cite{agv2}) as an algebra over the ordinary Chow ring of \mbun .

\subsection {Pull--backs}
\label{pullsection}
Let now $(Y,g)$  be a twisted sector of \mbun . Let $f: Y \to$ \mbun \ be the restriction to $Y$ of the natural map from the inertia stack to \mbun. In this section we study the pull--back morphism
\begin{displaymath}f^*: H^*(\overline{\mathcal{M}}_{1,n},\mathbb{Q}) \rightarrow H^*(Y, \mathbb{Q}).\end{displaymath}
The main results of this section are: 
\begin {enumerate}
\item the explicit description of the pull--back of the divisor classes of \mbun;
\item the pull--back morphism $f^*$ is determined by its restriction to the subalgebra of the cohomology generated by the divisors (Proposition \ref{subalgebra}). 
\end {enumerate}

Point (1) is enough for proving the generation part of Theorem \ref{prodfund}. Anyway it is only as a consequence of point (2) that we know all relations of $H^*_{CR}(\overline{\mathcal{M}}_{1,n})$ as an $H^*(\overline{\mathcal{M}}_{1,n})$-module (and, together with those of Section \ref{prodfund}, all relations as an $H^*(\overline{\mathcal{M}}_{1,n})$-algebra). 

Let us fix a twisted sector $(Y,g)=(\overline{Z}^{I_1, \ldots, I_k},g)$ (see Theorem \ref{partialtwistedcompact}); and identify it with the product of $\overline{Z} \times$ \mb{0}{I_1+1} $\times \ldots \times $ \mb{0}{I_k+1}. Therefore we have\footnote{see Lemma \ref{cycle} for the next equality} \begin{displaymath} A^*\left(\overline{Z}^{I_1,\ldots, I_k}\right)=H^{2*}(\overline{Z}^{I_1,\ldots, I_k})=A^*(\overline{Z}) \times A^*(\overline{\mathcal{M}}_{0,I_1+1})\times \ldots \times A^*(\overline{\mathcal{M}}_{0,I_k+1}).\end{displaymath} We call $p, p_1,\ldots p_{k}$ the projections onto the factors.  

The relations are obtained as follows
\begin{itemize}
\item take all divisor classes that are in the kernel of $f^*$,
\item for all $1 \leq i \leq k$, and each relation $R$ in $A^1(\overline{\mathcal{M}}_{0,I_i+1})$ (see \cite{keel}), take all linear combinations of divisor classes that pull-back via $f$ to $R$,
\item take any finite set generating $H^{odd}($\mbun$)$,
\item take all the subbanana cycles (see Definition \ref{bananacycle}).
\end{itemize}

Let us first deal with point (1), and thus compute the pull--back morphism for the divisor classes. The notation for the divisors in \mbun \ was introduced in Notation \ref{divisorigenere0e1}.
The pull--back $f^*(d_{irr})$ is zero when the base space $Z$ is a point. Otherwise,

\begin {enumerate}
\item it is $\frac{1}{2}[pt] \times[$\mb{0}{n+1}$]$, when the space is \azero;
\item it is $\frac{3}{2}[pt] \times[$\mb{0}{I_1+1}$]\times[$\mb{0}{I_2+1}$]$, when the space is \auno$^{I_1,I_2}$; 
\item it is $3 [pt] \times[$\mb{0}{I_1+1}$]\times[$\mb{0}{I_2+1}$]\times[$\mb{0}{I_3+1}$]$, when the space is \adue$^{I_1,I_2,I_3}$;
\item it is $3 [pt] \times [$\mb{0}{I_1+1}$]\times[$\mb{0}{I_2+1}$]\times[$\mb{0}{I_3+1}$]\times[$\mb{0}{I_4+1}$]$, when the space is \atre$^{I_1,I_2,I_3,I_4}$.
\end {enumerate}
This is a simple consequence of Theorem \ref{basedivisori}.

 The pullback $f^*([d_M])$ is zero when $M$ is not contained in any of the $I_i$s. Assume now that the base space $Z$ is a point. If $M$ is contained in (wlog) $I_1$, then there are two cases. If $M$ is a proper subset of $I_1$, then \begin{displaymath}f^*([d_M]) \cong [Z] \times  \Delta_M \times [\overline{\mathcal{M}}_{0, I_2+1}] \times \ldots \times [\overline{\mathcal{M}}_{0, I_k+1}].\end{displaymath} Otherwise, if $I=M$, then \begin{displaymath}f^*([d_M]) = p_1^*(-\psi_{I+1}).\end{displaymath}
Finally, if $\overline{Z}$ is one of the one-dimensional spaces $\overline{A_i}$, then the pullback $f^*([d_M])$ is computed by applying Theorem \ref{basedivisori}, similarly to $f^*(d_{irr})$.

A very important theoretical result follows as a corollary of our description of the pull--back morphism.

\begin {corollary} \label{surge} The morphisms $f^*: R^*($\mb{1}{n}$, \mathbb{Q}) \rightarrow A^*(Y, \mathbb{Q})$ are surjective. The same holds for the induced map in cohomology.
\end {corollary}
\begin {proof} 
Thanks to Lemma \ref{cycle}, it is sufficient to prove that the morphism $f^*: R^1($\mb{1}{n}$) \rightarrow A^1(Y)$ is surjective. The Kunneth decomposition reduces the problem to proving that one can obtain all divisors of each single factor of each twisted sector by pull--back from $R^1($\mb{1}{n}$)$. The above discussion shows that the set of divisor classes $\{d_{irr}, d_M\}_{M \subset [n]}$ surjects via $f^*$ onto $A^1(Y, \mathbb{Q})$. 

\end {proof}
\noindent Here is another way to express this result.
\begin {corollary} \label{cyclic}If $Y$ is a twisted sector, then the cohomology $H^*(Y, \mathbb{Q})$ is an $H^*($\mbun$, \mathbb{Q})$-module generated by the fundamental class $[Y]$.\footnote{Indeed $H^*(Y, \mathbb{Q})$ is cyclic also as an $R^*($\mbun$)$-module, or as a module over the subring of the tautological ring generated by the divisors.}
\end {corollary}

Following the plan established at the beginning of this section, we now deal with point (2). As a combination of \ref{conseteoremstar}, \ref{getzlerclaim} and \eqref{sub}, we have that all cohomology classes of \mbun \ can be written as sums of odd cohomology classes, of subbanana cycles (Definition \ref{bananacycle}), and of products of divisors. This leads to the decomposition
\begin{equation}
H^*(\overline{\mathcal{M}}_{1,n})=H^{odd}(\overline{\mathcal{M}}_{1,n}) \oplus \left(R^{div}_{1,n} + R^{ban}_{1,n} \right),
\end{equation}
where the sum of the last two terms is the vector space generated by boundary strata classes.
It is clear that the pull--back morphism $f^*$ is zero when restricted to the odd cohomology classes. We now prove that $f^*$ is zero when restricted to the vector subspace $R^{ban}_{1,n}$ generated by the subbanana cycles. Using Corollary \ref{corbasedivisori} we explicitly express $Y$ as a linear combination of product of divisor classes in \mbun. The product of a subbanana cycle with all the summands that contain a factor $d_{irr}$ is zero, because $d_{irr}^2=0$, and because a subbanana cycle is, in particular, a closed substack of $d_{irr}$. All the other summands have product zero with the subbanana cycles, because the set theoretic intersection of the substacks of \mbun \ that they describe is empty. We can now conclude.

\begin {proposition} \label{subalgebra} The pull--back morphism $f^*$ is determined by its restriction to the subalgebra of $H^*($\mbun$,\mathbb{Q})$ generated by the divisors.
\end {proposition}

\subsection{Push--forwards}
We now start the study of the push--forward morphism.
Let
\begin{displaymath}
 g:Z \to Y, \quad f:Y \to X 
\end{displaymath}
be respectively the inclusion of a double twisted sector in a twisted sector and of a twisted sector inside $X=$ \mbun. We study the push--forward morphism induced in cohomology by $f$ and $g$. Here is an easy corollary of Corollary \ref{corbasedivisori} and Corollary \ref{surge}.

\begin {corollary} \label{push} The push--forward morphism \begin{displaymath}f_{*}: A^*(Z, \mathbb{Q}) \rightarrow A^*(\overline{\mathcal{M}}_{1,n}, \mathbb{Q})\end{displaymath} has image in the tautological ring. The same holds for the push--forward map in cohomology.
\end {corollary}

Lemma \ref{surge} and Corollary \ref{push} make it possible to define an orbifold tautological ring in genus $1$.
\begin {definition} \label{tautolo} We define the \emph{orbifold tautological ring in genus $1$}
\begin {itemize}
\item Let $R^*_{CR}(\overline{\mathcal{M}}_{1,n})$ be defined as $R^*(\overline{\mathcal{M}}_{1,n}) \oplus \bigoplus H^*((X_i,g_i),  \mathbb{Q})$ as a vector space, where $X_i$ are all the twisted sectors and the grading is inherited from $H^*_{CR}(\overline{\mathcal{M}}_{1,n}, \mathbb{Q})$;
\item the product is the product $*_{CR}$ restricted to this rationally graded vector space.
\end {itemize}
\end {definition}
\noindent Note that, as a consequence of Theorem \ref{getzlerclaim},  $R^*_{CR}(\overline{\mathcal{M}}_{1,n})$ is a Gorenstein ring (a Poincar\'e duality ring).

We show how $g_*([Z])$ can be obtained as a pull--back of a class in $X$ in a canonical way.

\begin {definition} \label{clambda} If $\lambda=4$ or $\lambda=6$, we define $C_{\lambda}^*$ via the following pull--back diagram
\begin{displaymath}\xymatrix{C_{\lambda}^* \ar@{^{(}->}[r] \ar@{}|{\square}[dr] \ar[d] & \overline{\mathcal{M}}_{1,n} \ar[d]^{\pi_1} \\
C_{\lambda} \ar@{^{(}->}[r] & \overline{\mathcal{M}}_{1,1}.
}
\end{displaymath}
\end {definition}

\noindent Note that the equality:
$
[C_{\lambda}^*]= \frac{2}{\lambda} d_{irr}
$ holds in the tautological ring of \mbun.

\begin {proposition} \label {choice} With the notation introduced in this section, for $Z$ a double twisted sector and $Y$ a twisted sector, there is a canonical choice of $W$ closed substack of \mbun, such that $g_*([Z])= f^*([W])$.

\end {proposition}
\begin {proof} The only cases are, thanks to Proposition \ref{doubletwisted}:
\begin {enumerate}
\item it happens that $Z=Y$ or $Y=X$. In all these cases we choose $W:=X$; 
\item either $Z=C_{\lambda}^{[n]}$ for $\lambda=4,6$ and $Y=$\azero. In these cases we choose $W:=C_{\lambda}^*$;
\item or $Z=C_{\lambda}^{I_1,I_2}$ and $Y=$\auno$^{I_1,I_2}$. In these cases we choose $W=C_{\lambda}^*$. 
\end {enumerate}
One can easily check that these are all the cases that occur, and that the intersections are transversal.
\end {proof}

We have just fixed the cohomology classes that represent via pull--back the push--forward of all the fundamental classes. This choice determines the top Chern class of the excess intersection bundles via the projection formula.

\begin {corollary} \label{choicecorollary} Let now $E$ be the excess intersection bundle over the double twisted sector $Z$. Once the choice of \ref{choice} is fixed, a cohomology class $\beta$ on \mbun \ is determined such that:

\begin{displaymath}
g_*(c_{top}(E))=f^*(\beta).
\end{displaymath}
\end {corollary}
\begin {proof} If the top Chern class of $E$ is zero, we fix $\beta$ to be zero. When the top Chern class is $1$, the choice of Proposition \ref{choice} determines the class $\beta$ of this corollary too. The list of non trivial top Chern classes of excess intersection bundles (non-zero and not $1$), is given in \ref{topchern}. So, if the top Chern class is a $\psi$ class, there are only two possibilities: either the double twisted sector $Z$ is isomorphic to the twisted sector $Y$ (case $1$ of Corollary \ref{topchern}), or $Z=C_{\lambda}^{[n]}$ and $Y=\overline{A_1}^{[n]}$ (cases $2$ and $3$ of Corollary \ref{topchern}). In the first case, we choose $\beta$ to be $d_{[n]}$, the divisor with all markings on the genus $0$ component, and in the second case we choose \begin{displaymath}\beta:= d_{[n]} \cup [C_{\lambda}^*],
\end{displaymath}
\noindent where $C_{\lambda}^*$ is defined in \ref{clambda}.
\end {proof}
\subsection {Products of the fundamental classes of the twisted sectors}
\label{prodfund}
 If $X_i,X_j$ are twisted sectors, in this section we compute all the products $1_{X_i} *_{CR} 1_{X_j}$. In the perspective of Theorem \ref{fondamentale}, this section and Section \ref{pullsection}, provide the relations of $H^*_{CR}(\overline{\mathcal{M}}_{1,n})$ as an $H^*(\overline{\mathcal{M}}_{1,n})$-algebra generated by the fundamental classes of the twisted sectors. 

\begin {remark} \label {precedente} An explicit computation of all intersections of the twisted sectors, shows that besides the orbifold intersections of the kinds $(1_{X_i},\alpha) *_{CR} (1_{X_i}, \beta)$, and besides the trivial products $1_{\overline{\mathcal{M}}_{1,n}} *_{CR} 1_{X_i}$, the only pairs of twisted sectors whose fundamental classes give rise to non-zero Chen--Ruan products are in the following list:
\begin {enumerate}
\item $(($\azero$,-1), (C_4^{[n]},i/-i))$,
\item $(($\azero$,-1), (C_6^{[n]},\epsilon/ \epsilon^2/\epsilon^4/\epsilon^5))$, 
\item $(($\auno$^{I_1,I_2},-1), (C_4^{I_1,I_2}, i/-i))$.
\end {enumerate}
\end {remark}

\noindent We now compute the products of the pairs just described. Here if  $(X, \alpha)$ is a twisted sector, we write $H^*((X, \alpha), \mathbb{Q})$, which is a direct summand of $H^*_{CR}($\mbun$, \mathbb{Q})$ with its own grading. In other words, we implicitly assume the inclusion \begin{displaymath}i: H^*((X,\alpha), \mathbb{Q}) \subset H^*_{CR}(\overline{\mathcal{M}}_{1,n}, \mathbb{Q})\end{displaymath}
shifts the degree by twice the age of $(X, \alpha)$.
As usual, $-p_1^*(\psi_{\bullet})$ is the Chern class of Corollary \ref{topchern} and Theorem \ref{topchernt}.

\begin {corollary} \label{primo} With our usual notation for the twisted sectors, and with the notation introduced above, here is the explicit result of all the Chen--Ruan products described in Remark \ref{precedente}:
\begin {enumerate}
\item $[(C_4^{[n]}, i)] *_{CR} [(\overline{A_1}^{[n]},-1)]= p_1^*(-\psi_{\bullet}) \ \in H^2((C_4^{[n]},-i), \mathbb{Q})$,

\item $[(C_4^{[n]}, -i)] *_{CR} [(\overline{A_1}^{[n]},-1)]= [C_4^{[n]}] \ \in H^0((C_4^{[n]},i), \mathbb{Q})$,

\item $[(C_6^{[n]}, \epsilon)] *_{CR} [(\overline{A_1}^{[n]},-1)]= p_1^*(-\psi_{\bullet}) \ \in H^2((C_6^{[n]}, \epsilon^4), \mathbb{Q})$,

\item $[(C_6^{[n]}, \epsilon^2)] *_{CR} [(\overline{A_1}^{[n]},-1)]= p_1^*(-\psi_{\bullet}) \ \in H^2((C_6^{[n]},\epsilon^5), \mathbb{Q})$,

\item $[(C_6^{[n]}, \epsilon^4)] *_{CR} [(\overline{A_1}^{[n]},-1)]= [C_6^{[n]}] \ \in H^0((C_6^{[n]},\epsilon), \mathbb{Q})$,

\item $[(C_6^{[n]}, \epsilon^5)] *_{CR} [(\overline{A_1}^{[n]},-1)]= [C_6^{[n]}] \ \in H^0((C_6^{[n]},\epsilon^2), \mathbb{Q})$,

\item $[(\overline{A_2}^{I_1,I_2},-1)] *_{CR} [(C_4^{I_1,I_2},i)]=0 \in H^2((C_4^{I_1,I_2},-i),\mathbb{Q})$,

\item $[(\overline{A_2}^{I_1,I_2},-1)] *_{CR} [(C_4^{I_1,I_2},-i)]= [C_4^{I_1,I_2}] \in H^2((C_4^{I_1,I_2},i),\mathbb{Q})$,
\end {enumerate}
\end {corollary}

\begin {corollary} With our usual notation for the twisted sectors, and with the notation introduced above, here we recollect all the products of the kind $(1_{X_i},\alpha) *_{CR} (1_{X_i}, \beta)$: 

\begin{enumerate}

\item $[(X, \alpha)] *_{CR} [(X, \alpha^{-1})]= [X] \ \in H^*($\mbun$, \mathbb{Q})$,

\item $[(C_4^{[n]}, i)] *_{CR} [(C_4^{[n]}, i)] =  p_1^*(-\psi_{\bullet}) \cap [C_4^{[n]}] \ \in H^4((\overline{A_1}^{[n]},-1), \mathbb{Q})$,
\item $[(C_4^{[n]}, -i)] *_{CR} [(C_4^{[n]}, -i)] = [C_4^{[n]}] \ \in H^2((\overline{A_1}^{[n]},-1), \mathbb{Q})$,

\item $[(C_4^{I_1,I_2}, i)] *_{CR} [(C_4^{I_1,I_2}, i)] = 0 \ \in H^4((\overline{A_2}^{I_1,I_2},-1), \mathbb{Q})$,
\item $[(C_4^{I_1,I_2}, -i)] *_{CR} [(C_4^{I_1,I_2}, -i)] = [C_4^{I_1,I_2}] \ \in H^2((\overline{A_2}^{I_1,I_2},-1), \mathbb{Q})$,

\item $[(C_6^{[n]},\epsilon)] *_{CR} [(C_6^{[n]},\epsilon)]=0 \in H^4((C_6^{[n]},\epsilon^2), \mathbb{Q})$,

\item $[(C_6^{[n]},\epsilon)] *_{CR} [(C_6^{[n]},\epsilon^2)]= p_1^*(-\psi_{\bullet}) \cap [C_6^{[n]}] \ \in H^4((\overline{A_1}^{[n]},-1), \mathbb{Q})$,

\item $[(C_6^{[n]},\epsilon)] *_{CR} [(C_6^{[n]},\epsilon^4)]=0 \in H^4((C_6^{[n]},\epsilon^5), \mathbb{Q})$,

\item $[(C_6^{[n]},\epsilon^2)] *_{CR} [(C_6^{[n]},\epsilon^2)]=p_1^*(-\psi_{\bullet}) \in H^2((C_6^{[n]},\epsilon^4), \mathbb{Q})$,

\item $[(C_6^{[n]},\epsilon^2)] *_{CR} [(C_6^{[n]},\epsilon^5)]=p_1^*(-\psi_{\bullet}) \in H^2((C_6^{[n]},\epsilon), \mathbb{Q})$,

\item $[(C_6^{[n]},\epsilon^4)] *_{CR} [(C_6^{[n]},\epsilon^4)]=0 \in H^2((C_6^{[n]},\epsilon^2), \mathbb{Q})$,

\item $[(C_6^{[n]},\epsilon^4)] *_{CR} [(C_6^{[n]},\epsilon^5)]=[C_6^{[n]}] \ \in H^2((\overline{A_1}^{[n]},-1), \mathbb{Q})$,

\item $[(C_6^{[n]},\epsilon^5)] *_{CR} [(C_6^{[n]},\epsilon^5)]=[C_6^{[n]}] \ \in H^0((C_6^{[n]},\epsilon^4), \mathbb{Q})$,

\item $[(C_6^{I_1,I_2,I_3}, \epsilon^2)] *_{CR} [(C_6^{I_1,I_2,I_3}, \epsilon^2)]= 0 \ \in H^4((C_6^{I_1,I_2,I_3}, \epsilon^4), \mathbb{Q})$ and the result still holds when $I_3= \emptyset$,

\item $[(C_6^{I_1,I_2,I_3}, \epsilon^4)] *_{CR} [(C_6^{I_1,I_2,I_3}, \epsilon^4)]= 0 \ \in H^2((C_6^{I_1,I_2,I_3}, \epsilon^2), \mathbb{Q})$ and the result still holds when $I_3= \emptyset$,

\end {enumerate}
Moreover, the product of two fundamental classes of twisted sectors that do not belong to this list, nor to the one of Corollary \ref{primo}, is zero.
\end {corollary}

\begin{center} {\bf{Acknowledgments}} \end{center} I would like to acknowledge my advisor, Barbara Fantechi, for suggesting to me the topic of this work and for helpful discussions and suggestions. 
I would like to thank Gilberto Bini, Carel Faber and Orsola Tommasi for sharing with me their ideas devoted to further developments of the present work. A special thank to Sebastian Krug, who discovered some mistakes in the previous version of this preprint, and to the anonymous referee, who improved the quality of the manuscript. 
This work was supported by \texttt{SISSA}, \texttt{GNSAGA}, the Clay Mathematical Institute and \texttt{KTH}, Royal Institut of Technology

\begin{center}
Nicola Pagani,\\
Leibniz Universit\"at Hannover,
Welfengarten 1,\\
D-30167 Hannover, Germany.\\
e-mail: npagani@math.uni-hannover.de

\end{center}

\end{document}